\documentclass[12pt,a4paper,twoside,reqno]{amsart}

\usepackage{tikz}
\usepackage{slashed}
\usepackage{amsmath,amscd}
\usepackage{amssymb,amsfonts,mathrsfs}
\usepackage[driver=pdftex,margin=3cm,heightrounded=true,centering]{geometry}
\usepackage{JK}
\usepackage[colorlinks=true,linkcolor=blue,citecolor=blue]{hyperref}

\newtheorem{thm}{Theorem}[section]

\usepackage{graphicx}

\begin{document}
\title{On a spectral flow formula for the homological index}

\author{Alan Carey}
\address{Mathematical Sciences Institute,
Australian National University,
Canberra 0200, Australia}
\email{alan.carey@anu.edu.au}

\author{Harald Grosse}
\address{Department of Physics, University of Vienna, Boltzmanngasse, Vienna}
\email{harald.grosse@univie.ac.at}

\author{Jens Kaad}
\address{Institute for Mathematics, Radboud University, 6525 AJ Nijmegen, The Netherlands}
\email{jenskaad@hotmail.com}

\thanks{The first author thanks the Alexander von Humboldt Stiftung and colleagues at the University of M\"unster and acknowledges the support of the Australian Research Council. The third author acknowledges the support of the Radboud Excellence Initiative. All authors thank Fritz Gesztesy for discussions, the Hausdorff Institute for Mathematics and the Erwin Schr\"odinger Institute where some of this research was carried out. }

\begin{abstract}
Consider a selfadjoint unbounded operator $D$ on a Hilbert space $H$ and a one parameter norm
continuous family of selfadjoint bounded operators $\{A(t)\ \vert \ t\in \mathbb R\}$ that converges in norm to asymptotes $A_\pm$ at $\pm \infty$.  Then under certain conditions \cite{RS95} that include the assumption that 
the operators $\{D(t)= D+A(t), t\in \mathbb R\}$ all have discrete spectrum then the spectral flow along
the path $\{ D(t)\}$ can be shown to be equal to the index of $\partial_t+D(t)$ when the latter is an unbounded
Fredholm operator on $L^2(\mathbb R, H)$. In \cite{GLMST11} an investigation of the index=spectral flow question when
the operators in the path may have some essential spectrum was started but under restrictive assumptions that rule out differential operators in general.  In \cite{CGPST14a} the question of what happens when the Fredholm condition is dropped altogether was investigated. In these circumstances the Fredholm index is replaced by the Witten index. 

In this paper we take the investigation begun in \cite{CGPST14a} much further.
We show how to generalise a formula known from the setting of the $L^2$ index theorem
to the non-Fredholm setting.  Restricting back to the case of
 selfadjoint Fredholm operators our formula extends the result of \cite{RS95} in the sense of relaxing the discrete spectrum condition. It also generalises some other Fredholm operator results of \cite{Pu08, GLMST11, CGPST14a} that permit essential spectrum for the operators in the path. 
Our result may also apply however when the operators $\{D(t)\}$ have
essential spectrum equal to the whole real line. 

Our main theorem gives a trace formula relating the 
homological index of \cite{CaKa:TIH} to an integral formula that is known, for a path of selfadjoint Fredholms with compact resolvent and
with unitarily equivalent endpoints, to compute spectral flow. Our formula however, applies to paths of selfadjoint non-Fredholm operators.
We interpret this as indicating there is a generalisation of spectral flow to the non-Fredholm setting.
\end{abstract}

\maketitle
\tableofcontents

\section{Introduction}

\subsection{Motivation}

The issue of the relationship between spectral flow and the Fredholm index was first raised in \cite{APS75} and settled in the most definitive fashion for certain families of selfadjoint unbounded operators with compact resolvent in \cite{RS95}. For differential operators on noncompact manifolds it is typically the case that they possess some essential spectrum. An extension of the result of \cite{RS95} to this situation and its relationship to scattering theory was initiated in \cite{GLMST11} following \cite{Pu08}. However the key assumption in \cite{GLMST11} is that they consider spectral flow between selfadjoint operators that differ by a relatively trace class perturbation.  This latter assumption is violated in general for differential operators (although not for pseudodifferential operators). Indeed, there is comparatively little work available in the way of index formulas for operators with essential spectrum except for \cite{BGGSS:WKS}, \cite{BMS88} and previous work by the authors and their collaborators. This history is referenced in
 \cite{CGPST14a} which also contains results on an index theory for certain non-Fredholm operators using the model operator formalism of \cite{GLMST11}. 

More generally, the model operators of \cite{RS95} provide prototypes  for various applications  such as those arising in connection with the Maslov index, Morse theory, Floer homology, Sturm oscillation theory, etc. The 
principle aim in \cite{Pu08} and \cite{GLMST11} was to extend the  results in \cite{RS95}
(albeit subject to  a relatively trace class perturbation condition), in a fashion permitting essential spectra. This work has motivated the current investigation where we provide evidence that the relationship between the Fredholm index and spectral flow is a special case of an operator trace identity that holds without any
Fredholm assumptions. The fundamental trace formula of \cite{GLMST11} fits into the general form of the identity we prove here.

The prototype of our formula also appears implicitly in \cite{BCPRSW} where it is proved in the context of an
$L^2$ index theorem for covering spaces.  To simplify the discussion and to explain the formula
we now describe only the special case of \cite{BCPRSW} that applies to compact manifolds.

\subsection{The index=spectral flow formula}
We use the notation of \cite{BCPRSW} as the setting differs from that considered in this 
current paper. Assume that we have an odd dimensional closed manifold $M$.

The analysis of spectral flow of a path of Dirac-type operators on sections of bundles over $M$ traditionally proceeds by replacing $M$ by $M\times S^1$ or $M\times [0,1]$ and considering the Dirac operator on this even dimensional manifold.

The argument in \cite{BCPRSW} uses an adiabatic limiting process, which leads to the formula for the leading 
term in the expansion of the (difference of) heat kernels. In this paper we also use an analogous adiabatic limiting process to establish our main result. 

Let $B : C^{\infty}(M;S) \to C^{\infty}(M;S)$ denote a Dirac-type operator acting on sections of a bundle 
of Clifford modules $S$ over a closed, odd-dimensional manifold $M$ as in \cite{BW}. Introduce an auxiliary hermitian vector bundle 
$E$ with hermitian connection $\nabla$, and the operator
$B_0 = B \otimes_{\nabla}\T{Id}_E$. Let $g : E \to E$ denote a unitary bundle automorphism, then we can introduce the operator
\[
B_1 = gB_0g^{-1} = (\T{Id} \otimes g) (B \otimes_{\nabla}\T{Id}_E)(\T{Id} \otimes g^{-1})
\]
The difference $C=B_1 - B_0 = [g,B]g^{-1}$ is a bundle endomorphism and we want to present a formula 
for the spectral flow along a path joining $B_0$ and $B_1$.

Spectral flow is a homotopy invariant and the space of bounded perturbations of $B_0$ is affine. So in this example we may restrict ourselves to the study of the spectral flow of a smooth family of selfadjoint operators parametrised by $S^1$. We introduce a smooth cut-off function $\alpha : \mathbb R \to \mathbb R$, such that
\[
\alpha(s) \ = 
 \ \begin{cases} & 0 
\quad \ \text{if}
 \ \ \ s \le 1/4  \\
    &  1  \quad \ \text{if}
 \ \ 3/4 \le s \end{cases}
\]
We also assume that there exists a positive constant $c$, such that $\biggm|\frac{d^k\alpha}{ds^k}\biggm| \le c{\cdot}s$ 
for $0\le s \le 1$ , $k = 0,1,2$. Now we consider the family
\begin{equation}\label{e:2}
\{B_s = B_0 + \alpha(s)C\}
\end{equation}
which in an obvious way provides us with a family of operators on $S^1$. We may also consider the corresponding operator $\mathcal D = \partial_s + B_s$ on the closed manifold $S^1 \times M$ where $B_s$ is given by the formula (\ref{e:2}). The operator $\mathcal D$ acts on sections of the bundle $[0,1] \times S\otimes E/_{\cong}$, where the identification is given by
\[
(1,y;g(y)w) \cong (0,y;w) \ \ \text{where} \ \ 
w \in S_y \otimes E_y
\]

By rescaling the variable $s\to s/\delta$ in the part of the Schwartz kernel for the heat operators that arises from
$\partial_s$ and taking the (adiabatic) limit $\delta\to 0$, \cite{BCPRSW} establishes that the following formula holds,
\begin{equation}\label{e:3}
\mbox{index} \ \mathcal D = \T{Tr}_{M\times S^1}(e^{-\epsilon \mathcal D^\ast \mathcal D}-e^{-\epsilon \mathcal D \mathcal D^\ast})
 = \T{sf}\{B_s\} = \sqrt{\frac{\epsilon}{\pi}}\int_0^1
\T{Tr}_M( {\dot B}_s e^{-\epsilon B_s^2})\, ds
\end{equation}
where as usual ${\dot B}_s = \frac{dB}{ds}$.

As we remarked above the equality index=spectral flow  in (\ref{e:3}) goes back to the original Atiyah-Patodi-Singer paper \cite{APS75}. They also proved a formula $\T{sf}\{B_s\} = \int_0^1{ \eta}_sds$, where $\eta_s$ denotes the $\eta$-invariant of the operator $B_s$. The equality 
\[
\T{sf}\{B_s\} = \sqrt{\frac{\epsilon}{\pi}} \int_0^1\T{Tr}_M{\dot B}_se^{-\epsilon B_s^2}\, ds
\]
and more was proved by Wojciechowski around 1991 and published in one of his IUPUI preprints \cite{KPW1}. The formal paper never appeared but the result eventually resurfaced in the paper by Getzler \cite{G}. Getzler's paper motivated a careful investigation of such spectral flow formulas in \cite{CP1,CP2}. Note that to our knowledge the first mathematically precise formulation of the adiabatic limit is in \cite{Ch87}.

Using a Laplace transform argument of the kind employed in \cite{CP2}  we now see that for $r>0$ 
\[
\int_0^\infty \epsilon^re^{-\epsilon} (e^{-\epsilon \mathcal D^\ast\mathcal D}
-e^{-\epsilon \mathcal D\mathcal D^\ast})\, d\epsilon
=\Gamma(r) ((1+\mathcal D^*\mathcal D)^{-r}-(1+ \mathcal D \mathcal D^*)^{-r})
\]
while, inserting the same function of $\epsilon$ and integrating  on the RHS we obtain:
\[
\int_0^\infty \epsilon^{r+1/2}e^{-\epsilon}{\pi}^{-1/2}\int\frac{dB(s)}{ds} e^{-\epsilon B(s)^2}\, dsd\epsilon
\]
and because it is permissable to interchange the order of integration (as in \cite{CP2})  we obtain
\[
\Gamma({r+1/2})\pi^{-1/2}\int\frac{dB(s)}{ds} (1+B(s)^2)^{-(r+1/2)}\, ds
\]
Under the assumptions made here we may apply traces to these expressions and then interchange the
trace and integral (as in \cite{CP2}) to obtain for $r$ sufficiently large:
\begin{equation}\label{FTF}
\T{Tr}_{M\times S^1}((1+\mathcal D^\ast\mathcal D)^{-r}-(1+ \mathcal D\mathcal D^\ast)^{-r})
=C_{r+1/2}
\int \T{Tr}_M(\frac{dB(s)}{ds} (1+B(s)^2)^{-(r+1/2)})\, ds
\end{equation}
where $C_{r+1/2}=\Gamma({r+1/2})\pi^{-1/2}\Gamma(r)^{-1}$. We refer to this last equation as the {\it integral trace formula}.

{\it The objective of this paper is to show that this integral trace formula holds when we relax the 
condition that we are working on a compact manifold $M$. In this paper we will replace the manifold situation by an abstract approach.  
This is possible because we introduce and exploit resolvent techniques that are not
available in the heat kernel method. Our abstract approach includes examples on noncompact manifolds and makes no reference to the Fredholm property. The resolvent methods 
and `adiabatic approximation' we employ may
also  have wider application.}

In the notation above this means  that  while the individual operators 
\[
(1+\mathcal D^\ast\mathcal D)^{-r}, \ (1+ \mathcal D\mathcal D^\ast)^{-r}, \ (1+B(s)^2)^{-(r+1/2)}
\]
are no longer trace class for any $r$, the particular combinations of operators that appear under the trace in the trace formula remain
trace class and the formula will be proved to hold. What this result supports is our contention that the index=spectral flow theorem is a special case of an operator trace identity (the integral trace formula) that holds in great generality.

\subsection{Relation to earlier work}

The LHS of the integral trace formula (\ref{FTF}) was shown in \cite{CaKa:TIH} to have a (co)-homological interpretation. In \cite{CaKa:TIH} we studied the algebra generated by a pair $T, T^*$ satisfying
\begin{equation}\label{TC}
\T{Tr}[(1-TT^*)^n -(1-T^*T)^n]<\infty 
\end{equation} for some positive integer $n$.
Then we showed that the real number  (called the homological index there),
$\T{Tr}[(1-TT^*)^n -(1-T^*T)^n] $,
computes the pairing of a certain element of a homology group of the algebra generated by the operators
$T, T^*$ with an element of the dual cohomology group.

The relationship of this general result for pairs $T, T^*$ satisfying (\ref{TC})
to the integral trace formula arises via the mapping that sends unbounded operators
$\mathcal D$ to $T_{\mathcal D}=\mathcal D(1+\mathcal D^*\mathcal D)^{-1/2}$.
We show later that for the operators $\mathcal D$ studied in this paper
the corresponding operators $T_{\mathcal D}$ satisfy (\ref{TC}).
Thus the integral trace formula expresses the homological index in terms of an integral formula
that in the case of Fredholm operators with purely discrete spectrum is computing spectral flow.
This suggests that when we do not impose Fredholm or discrete spectrum conditions on the operators in the path  then the RHS of the integral trace formula is a generalisation of spectral flow.
Most importantly it allows the operators in the path to have essential spectrum.

We note that previous evidence has been gathered that supports this view. In \cite{GLMST11, CGPST14a}
 versions of the integral trace formula are proved.  Under very different hypotheses (which are too restrictive to support the analysis described  here)
in  \cite{GLMST11} the $r=1$ case is studied when the operators on both sides
of the integral trace formula are
Fredholm operators. There it is shown that the homological index satisfies
\[
\T{Tr}[(1-T_{\mathcal D}T_{\mathcal D}^*) -(1-T_{\mathcal D}^*T_{\mathcal D})]
=\T{Tr}[(1+\mathcal D^\ast\mathcal D)^{-1}-(1+ \mathcal D\mathcal D^\ast)^{-1}]
\]
while the scaling limit
$\lim_{\lambda\to 0} \T{Tr}[(1+\lambda^{-1}\mathcal D^\ast\mathcal D)^{-1}-(1+ \lambda^{-1}\mathcal D\mathcal D^\ast)^{-1}]$
is equal to the Fredholm index of $\mathcal D$.  From the version of the integral trace formula 
in \cite{GLMST11} it is established by an indirect argument that the RHS is related to spectral flow.
In the non-Fredholm case as studied in \cite{CGPST14a}, again under hypotheses that are far too restrictive for our purposes,
the limit
$\lim_{\lambda\to 0} \T{Tr}[(1+\lambda^{-1}\mathcal D^\ast\mathcal D)^{-1}-(1+ \lambda^{-1}\mathcal D\mathcal D^\ast)^{-1}]$
is shown to compute the so-called Witten index\footnote{For a full discussion of the Witten index, spectral shift functions and their relation to the Fredholm index see \cite{CGPST14a}.}  while the RHS of the integral trace formula is linked to Krein's spectral shift function for the pair of operators  given by taking the endpoints of the path.  It is not known at this time how to interpret this
spectral shift function as giving a generalisation of spectral flow to the situation where the path
does not consist of Fredholm operators. Work is in progress on low dimensional examples 
to investigate this.

\subsection{The setting and main result}\label{s:set}
For brevity we will write $H_1=L^2(\mathbb R)$, $H_2$ denotes a separable Hilbert space
and sometimes we use the Hilbert space tensor product $H=H_1\widehat{\otimes} H_2$ which is identified with $L^2(\mathbb R, H_2)$.
Throughout this text $D_2 : \sD(D_2) \to H_2$ will be a selfadjoint unbounded operator acting on  $H_2$.
We will apply the notation $\De_2 := D_2^2 : \sD(D_2^2) \to H_2$ for the square of $D_2 : \sD(D_2) \to H_2$. This unbounded operator is then both selfadjoint and positive.

We define the dense subspace
\[
H_2^\infty \su H_2 \q H_2^\infty := \cap_{k = 0}^\infty \sD(\De_2^{k/2})
\]
and notice that $H_2^\infty$ is a core for each of the positive unbounded operators $\De_2^{k/2} : \sD(\De_2^{k/2}) \to H_2$, $k \in \nn_0$.
For each bounded operator $T : H_2 \to H_2$ such that $T\big( H_2^\infty \big) \su H_2^\infty$ we denote the $k^{\T{th}}$ iterated commutator with $\De_2^{1/2}$ by
\[
\de_2^k(T) : H_2^\infty \to H_2 \q k \in \nn_0
\]
Thus, for example $\de_2^2(T) = \De_2 T + T \De_2 - 2 \cd \De_2^{1/2} T \De_2^{1/2} : H_2^\infty \to H_2$.

For each $q \in [1,\infty)$, we let $\sL^q(H_2)$ denote the $q^{\T{th}}$ Schatten ideal in the bounded operators $\sL(H_2)$.
Let us fix a selfadjoint bounded operator $A : H_2 \to H_2$.
The following assumption will be in effect throughout this text:

\begin{assu}\label{a:psesum}
\begin{enumerate}
\item $A$ is a \emph{pseudodifferential operator of order $0$} with respect to $\De_2 : \sD(\De_2) \to H_2$. This means,
$
A(H_2^\infty) \su H_2^\infty
$
and the iterated commutator
$
\de_2^k(A) : H_2^\infty \to H_2
$
extends to a bounded operator on $H_2$ for all $k \in \nn_0$.
\item There exists a $p_0 \in (0,\infty)$ 
such that
\[
(1 + \De_2)^{-r/2} A (1 + \De_2)^{-s/2} \in \sL^q(H_2)
\]
for all $r, s \in [0,\infty)$ with $r + s > 0$ and all $q \in [1,\infty) \cap \big(p_0/(r + s), \infty\big)$.
\end{enumerate}
\end{assu}

\begin{remark}
We are working within the framework of the Connes-Moscovici pseudodifferential calculus \cite{CoMo:LIF}. However we make no use of the notion of spectral triple and indeed contrary to the case of non-unital spectral triples as expounded in \cite{CGRS:ILN}, we will \emph{never} assume that the commutator
$
[D_2, A] : \sD(D_2) \to H_2
$
extends to a bounded operator on $H_2$.
\end{remark}

Throughout this text we let $p_0 \in (0,\infty)$ be fixed and chosen to satisfy the second condition in the above assumption.
Next we introduce a smooth function $h : \rr \to \rr$ with the property that there exists a $K \in \nn$ with 
\begin{equation}\label{eq:conasy}
h(t) = h(K) \T{ and } h(-t) = h(-K) \T{ for all } t \geq K
\end{equation}
We let $h_+ := h(K)$ and $h_- := h(-K)$.

Define the selfadjoint bounded operator 
\[
h \cd A : L^2(\rr,H_2) \to L^2(\rr,H_2)
\]
by $(h \cd A)(\xi)(t) = h(t) \cd A(\xi(t))$ for all $\xi \in C_c(\rr,H_2)$ and all $t \in \rr$. Here $L^2(\rr,H_2)$ denotes the Hilbert space of (equivalence classes) of square integrable functions from $\rr$ to $H_2$ and $C_c(\rr,H_2) \su L^2(\rr,H_2)$ denotes the dense subspace of continuous compactly supported functions from $\rr$ to $H_2$.

Let us also define the unbounded operator $\dir_+ : \sD(\dir_+) \to L^2(\rr,H_2)$ as the closure of
\[
\frac{d}{dt} \ot 1 + 1 \ot D_2 : C_c^\infty(\rr) \ot \sD(D_2) \to L^2(\rr,H_2)
\]
Then $\dir_+$ is normal with adjoint $\dir_- : \sD(\dir_-) \to L^2(\rr,H_2)$ given by the closure of
\[
- \frac{d}{dt} \ot 1 + 1 \ot D_2  : C_c^\infty(\rr) \ot \sD(D_2) \to L^2(\rr,H_2)
\]
We apply the notation $\lap := \dir_- \dir_+ = \dir_+ \dir_-$ and notice that $\lap$ agrees with the closure of
\[
- \big( \frac{d}{dt}\big)^2 \ot 1 + 1 \ot \De_2 : C_c^\infty(\rr) \ot \sD(\De_2) \to L^2(\rr,H_2)
\]

We form the closed unbounded operator
\[
D_+ := \dir_+ + h \cd A : \sD(\dir_+) \to L^2(\rr,H_2)
\]
and note that the adjoint is given by
\[
D_+^* := D_- := \dir_- + h \cd A : \sD(\dir_-) \to L^2(\rr,H_2) 
\]

To state the main theorem we introduce the notation
\[
A_+ := h_+ \cd A = h(\infty) \cd A \q A_- := h_- \cd A = h(-\infty) \cd A
\]
Our main result is the following:

\begin{thm}\label{t:intlarscaI}
Let $m \in \nn$ with $m > p_0/2$ be given. Then
\[
\begin{split}
&\T{Tr}\big( (\la + D_- D_+)^{-m} - (\la + D_+ D_-)^{-m} \big)\\
& \q = C_{m + 1/2}
\cd \int_0^1 \T{Tr}\Big( A_+ (\la + (D_2 + r \cd A_+)^2)^{-m-1/2} \\
& \qqq \qq \qq - A_- (\la + (D_2 + r \cd A_-)^2)^{-m-1/2} \Big) \, dr
\end{split}
\]
for all $\la > 0$ where the constant $C_{m + 1/2} > 0$ is given above and by
\[
C_{m + 1/2} := \frac{m}{\pi} \cd \int_{-\infty}^{\infty} (1 + \eta^2)^{-m-1} \, d\eta
\]
\end{thm}

\begin{remark} We note an intriguing corollary of our result.
Notice that the endpoints of the operators in the path on the right hand side may have a kernel in which case $D_+$ is not Fredholm \cite{CGPST14a}. In this instance the scaling limit as $\lambda \to 0$ on the LHS computes the Witten index \cite{CGPST14a}. When the operators in the path on the RHS have compact resolvent and the endpoints are unitarily equivalent then the RHS computes spectral flow \cite{ACS} (independently of the value of $\lambda$). Thus, in these circumstances, the homological index is constant in $\lambda$ and integral.
\end{remark}

\section{Preliminaries}

We begin by explaining some simple results that justify some later computations.
The following result is standard and will be stated without a proof. The interested reader may consult the following references for details, \cite[Appendix B]{CoMo:LIF}, \cite[Chapter 3]{Hig:LFN}.

\begin{prop}\label{p:regpro}
Let $g : \rr \to \rr$ be a smooth function such that the derivative $\frac{d^k g}{dt} : \rr \to \rr$ is bounded for all $k \in \nn_0$. Then the product $g \cd A : L^2(\rr,H_2) \to L^2(\rr,H_2)$ is a pseudodifferential operator of order $0$ with respect to $\lap : \sD(\lap) \to L^2(\rr,H_2)$. 
\end{prop}

There is a  summability condition (in the sense of spectral triples) in Assumption \ref{a:psesum} 
that is not standard in the literature, and we therefore give a careful proof of the following proposition.

\begin{prop}\label{p:sumpro}
Let $g : \rr \to \rr$ be a smooth compactly supported function. Then
\[
(\lap + 1)^{-r/2} g \cd A (\lap + 1)^{-s/2} \in \sL^q\big( L^2(\rr,H_2) \big)
\]
for all $r,s \in [0,\infty)$ with $r + s > 0$ and all $q \in [1,\infty) \cap \big( (p_0 + 1)/(s + r) , \infty \big)$.
\end{prop}
\begin{proof}
Let $\De_1 : \sD(\De_1) \to L^2(\rr)$ denote the closure of the Laplacian $-\frac{d^2}{dt} : C_c^\infty(\rr) \to L^2(\rr)$. Let $s \in (0,\infty)$ and let us show that
\begin{equation}\label{eq:trarea}
g (1 + \De_1)^{-s/2} \in \sL^q(L^2(\rr))
\end{equation}
for all $q \in [1,\infty) \cap \big( 1/s, \infty \big)$. By an application of the H\"older inequality this will imply that
\[
(1 + \De_1)^{-r/2} g (1 + \De_1)^{-s/2} \in \sL^q(L^2(\rr))
\]
for all $r,s \in [0,\infty)$ with $r + s > 0$ and all $q \in [1,\infty) \cap \big( 1/(s + r), \infty \big)$.

Suppose first that $s \in (0,1/2]$. In this situation we have that 
\[
g (1 + \De_1)^{-s/2} \in \sL^q(L^2(\rr))
\]
for all $q \in (1/s,\infty)$ by an application of \cite[Theorem 4.1]{Sim:TIA}.

Suppose now that $s \in (1/2,\infty)$ and let $q \in [1,\infty) \cap ( 1/s,\infty )$ be given. Since $\sL^{q_0}(L^2(\rr)) \su \sL^{q_1}(L^2(\rr))$ for $1 \leq q_0 \leq q_1$ we may suppose that $q \in [1,2] \cap (1/s,2]$. By \cite[Theorem 4.5]{Sim:TIA} we can conclude that $g (1 + \De_1)^{-s/2} \in \sL^q(L^2(\rr))$ provided that
\[
\sum_{n \in \zz} \big( \int_{n - 1/2}^{n + 1/2} (1 + t^2)^{-s} \, dt \big)^{q/2} < \infty
\]
But this follows from the estimates
\[
\sum_{n = 2}^\infty \big( \int_{n - 1/2}^{n + 1/2} (1 + t^2)^{-s} \, dt \big)^{q/2}
\leq \sum_{n = 2}^\infty  (1 + (n-1)^2)^{-sq/2}
\leq \sum_{n = 1}^\infty n^{-sq} < \infty
\]
where the last sum converges since $q > 1/s$.

Let now $r,s \in [0,\infty)$ with $r + s > 0$ be given and let $q \in [1,\infty) \cap \big( (p_0 + 1)/(s + r),\infty \big)$. We need to show that
\[
(\lap + 1)^{-r/2} g A (\lap + 1)^{-s/2} \in \sL^q\big( L^2(\rr,H_2) \big)
\]

In what follows, we will suppress tensor products with the identity operator from the notation. Thus, we will write $\De_1$ and $\De_2$ for the closures of
\[
\De_1 \ot 1 : \sD(\De_1) \ot H_2 \to L^2(\rr,H_2) \q \T{and} \q
1 \ot \De_2 : H_2 \ot \sD(\De_2) \to L^2(\rr,H_2)
\]
respectively.
We remark that the operators
\[
G_2 := (1 + \De_2)(\lap + 1)^{-1} \T{ and } G_1 := (1 + \De_1 )(\lap + 1)^{-1} : L^2(\rr,H_2) \to L^2(\rr,H_2)
\]
are bounded and positive. We remark also that these operators commute and that they also both commute with $(\lap + 1)^{-1}$, $(1 + \De_2)^{-1}$ and $(1 + \De_1)^{-1}$. We finally notice the relations
\[
G_2 \cd (1 + \De_2)^{-1} = (\lap + 1)^{-1} = G_1 \cd ( 1 + \De_1)^{-1}
\]

Let now
\[
s_1 := s/(1 + p_0) \q s_2 := (s p_0)/(1 + p_0) \q r_1 := r/(1 + p_0) \q r_2 := (rp_0)/(1 + p_0)
\]
Since $s_1 + s_2 = s$ and $r_1 + r_2 = r$ we obtain that
\[
\begin{split}
(\lap + 1)^{-r/2} g A (\lap + 1)^{-s/2}
& = G_1^{r_1/2} G_2^{r_2/2} \cd (1 + \De_2)^{-r_2/2} A (1 + \De_2)^{-s_2/2} \\
& \q \cd (1 + \De_1 )^{-r_1/2} g (1 + \De_1 )^{-s_1/2} \cd G_2^{s_2/2} G_1^{s_1/2}
\end{split}
\]

Recall next that $\sL^q\big( L^2(\rr) \big) \ot \sL^q( H_2 ) \su \sL^q\big( L^2(\rr,H_2) \big)$. It therefore suffices to show that
\[
\begin{split}
& (1 + \De_2)^{-r_2/2} A (1 + \De_2)^{-s_2/2} \in \sL^q(H_2) \q \T{and} \\
& (1 + \De_1)^{-r_1/2} g (1 + \De_1 )^{-s_1/2} \in \sL^q\big(L^2(\rr)\big)
\end{split}
\]
But this follows from the above since
$
p_0/(s_2 + r_2) = (p_0+1)/(s + r) = 1/(s_1 + r_1)
$
\end{proof}

\section{Existence of the homological index}\label{s:exihom}
Throughout this Section, the assumptions and the notation of Section \ref{s:set} will remain in effect.

{\bla
\emph{The following proposition shows that we have a well-defined \emph{homological index}:
\[
\T{H-Ind}_\la^m(D_+) := \la^m \cd \T{Tr}\big( (\la + D_- D_+)^{-m} - (\la + D_+ D_-)^{-m} \big)
\]
for all $\la > 0$ and all $m \in \nn$ with $m > (p_0 - 1)/2$. The main purpose of this paper is to compute this homological index in terms of an integral of a path of operators acting on the `inner' Hilbert space $H_2$.}}

\begin{prop}\label{p:exihom}
Let $\la > 0$ and $m \in \nn$ with $m > (p_0 - 1)/2$ be given. Then the difference
\[
(\la + D_- D_+)^{-m} - (\la + D_+ D_-)^{-m}
\]
is of trace class. 
\end{prop}
\begin{proof}
We are interested in applying \cite[Theorem 3.1]{CaGrKa:AOS}. We recall that this theorem shows that the difference
$
(\la + D_- D_+)^{-m} - (\la + D_+ D_-)^{-m}
$
is of trace class provided that the following conditions are satisfied:
\begin{enumerate}
\item $\dir_+ : \sD(\dir_+) \to L^2(\rr,H_2)$ is a normal unbounded operator.
\item The bounded operator $h A : L^2(\rr,H_2) \to L^2(\rr,H_2)$ is a pseudodifferential operator of order $0$ with respect to $\lap := \dir_- \dir_+$.  \item The sum of commutators 
\[
[\dir_+,h\cd A ] + [h \cd A,\dir_-] : \sD(\dir_+) \to L^2(\rr,H_2)
\]
extends to a bounded operator $F : L^2(\rr,H_2) \to L^2(\rr,H_2)$.
\item The bounded operator
\[
(\lap + 1)^{-j-1} F (\lap + 1)^{-m+ j}
\]
is of trace class for all $j \in \{0,\ldots,m-1\}$.
\end{enumerate}

We already remarked in Section \ref{s:set} that $\dir_+ : \sD(\dir_+) \to L^2(\rr,H_2)$ is a normal unbounded operator. Furthermore, it is implied by Proposition \ref{p:regpro} that condition $(2)$ is satisfied as well.
We now notice that the sum of commutators
\[
[\dir_+,h\cd A ] + [h \cd A,\dir_-] : \sD(\dir_+) \to L^2(\rr,H_2)
\]
extends to the bounded operator
\[
F := 2 \frac{d h}{dt} \cd A : L^2(\rr,H_2) \to L^2(\rr,H_2)
\]

The fact that condition $(4)$ is satisfied is now a consequence of Proposition \ref{p:sumpro}. Notice here that $\frac{dh}{dt} : \rr \to \rr$ has compact support since $h : \rr \to \rr$ has constant asymptotics, see \eqref{eq:conasy}.
\end{proof}

\section{Rescaling of the flow parameter}\label{s:resflopar}
In this Section we begin the `adiabatic limit' process and
throughout  we will apply the notation and the assumptions stated in Section \ref{s:set}.
For each $\ep > 0$ we define the function 
\[
h_\ep : \rr \to \rr \q h_\ep : t \mapsto h(\ep \cd t)
\]
We may then introduce the alternative closed unbounded operator
\[
D_+(\ep) = \dir_+ + h_\ep \cd A : \sD(\dir_+) \to L^2(\rr,H_2)
\]
which has the adjoint
\[
D_+(\ep)^* := D_-(\ep) := \dir_- + h_\ep \cd A : \sD(\dir_-) \to L^2(\rr,H_2)
\]

Since $h_\ep : \rr \to \rr$ is smooth and has a compactly supported derivative $\frac{d h_\ep}{dt} : \rr \to \rr$ we then have an alternative homological index:
\[
\T{H-Ind}^m_\la\big( D_+(\ep) \big) = \T{H-Ind}_\la^m\big( \dir_+ + h_\ep \cd A \big)
\]
whenever $\la > 0$ and $m \in \nn$ with $m > (p_0 - 1)/2$, see Proposition \ref{p:exihom}.

\emph{The aim of this Section is to prove, for $m \in \nn$ with $m > p_0/2$, the identity
\[
\T{H-Ind}^m_\la\big( D_+(\ep) \big) = \T{H-Ind}^m_\la(D_+).
\]
 This will rely on the topological invariance result for the homological index which was obtained in \cite[Theorem 8.1]{CaKa:TIH}.}
\medskip

In order to apply our invariance result, we define the selfadjoint bounded operator
\[
B_\ep := \ma{cc}{0 & (h_\ep - h) A \\ (h_\ep - h) A & 0} : L^2(\rr,H_2)^2 \to L^2(\rr,H_2)^2
\]
and the selfadjoint unbounded operators
\[
D_t := \ma{cc}{0 & D_- \\ D_+ & 0} + t B_\ep : \sD(\dir_+)^2 \to L^2(\rr,H_2)^2 \q t \in [0,1]
\]
Thus $D_0 = \ma{cc}{0 & D_- \\ D_+ & 0}$ and $D_1 = \ma{cc}{0 & D_-(\ep) \\ D_+(\ep) & 0}$.

We emphasize that the difference $h_\ep - h : \rr \to \rr$ is a smooth \emph{compactly supported} function.
Recall the notation $H := L^2(\rr,H_2)$ and write $H^\infty := \cap_{k = 0}^\infty \sD(\lap^{k/2}) \su H$.

\begin{thm}\label{t:adihomind}
Let $m \in \nn$ with $m > p_0/2$. Then
\[
\T{H-Ind}_\la^m(D_+(\ep)) = \T{H-Ind}_\la^m(D_+)
\]
for all $\ep > 0$ and all $\la > 0$. In particular, we have that
\[
\T{H-Ind}_\la^m(D_+) = \lim_{\ep \searrow 0} \T{H-Ind}_\la^m(D_+(\ep))
\]
for all $\la > 0$.
\end{thm}
\begin{proof}
Let $\ep > 0$ be given. According to the conditions of \cite[Theorem 8.1]{CaKa:TIH} we need to show that
\begin{enumerate}
\item The map 
$
t \mapsto T_t := D_t (1 + D_t)^{-1/2}; \  [0,1] \to \sL(H \op H)
$
is continuously differentiable in operator norm.
\item The bounded operator
\[
(1 + D_0^2)^{-j} B_\ep (1 + D_0^2)^{-k-1/2} : H \op H \to H \op H
\]
lies in the Schatten ideal $\sL^{m/(j+k)}( H \op H )$ for all $j,k \in \{0,\ldots,m\}$ with $j + k \in \{1,\ldots,m\}$.
\item There exists a $\de \in (0,1/2)$ such that 
\[
B_\ep (1 + D_0^2)^{-m - 1/2 + \de} \in \sL^1( H \op H )
\]
\end{enumerate}

The first of these conditions is a consequence of \cite[Appendix A, Theorem 8]{CP1}. 
To verify the second and the third condition we notice that we may apply \cite[Lemma 2.6]{CGPRS:SNS} and Proposition \ref{p:regpro} to conclude that the following holds for each $s \in [0,\infty)$:
\begin{enumerate}
\item The bounded operator $(1 + D_0^2)^{-s}: H \op H \to H \op H$ has the dense subspace $H^\infty \op H^\infty \su H \op H$ as an invariant subspace.
\item The unbounded operator
\[
\ma{cc}{(\lap + 1)^s & 0 \\ 0 &  (\lap + 1)^s}(1 + D_0^2)^{-s} : H^\infty \op H^\infty \to H \op H
\]
extends to a bounded operator on $H \op H$.
\end{enumerate}

Let now $j,k \in \{0,\ldots,m\}$ be given with $j + k \in \{1,\ldots,m\}$. In view of the above observation it now suffices to show that
\[
\T{diag}\big(  (\lap + 1)^{-j} \big) \cd B_\ep \cd \T{diag}( (\lap + 1)^{-k-1/2})
\in \sL^{m/(j + k)}(H \op H)
\]
and furthermore that
\[
B_\ep \cd \T{diag}\big( (\lap + 1)^{-m - 1/2 + \de} \big) \in \sL^1(H \op H)
\]
for some $\de \in (0,1/2)$. Here $\T{diag}( (\lap + 1)^{-j})$ is notation for the two by two diagonal matrix with $(\lap + 1)^{-j}$ in both of the diagonal entries. 

The above conditions are however an immediate consequence of Proposition \ref{p:sumpro} since $h - h_\ep : \rr \to \rr$ is a smooth compactly supported function.
\end{proof}

\section{First approximation of the homological index}\label{s:firapphom}
Let us recall that $H := L^2(\rr,H_2)$ and $H^\infty := \cap_{k = 0}^\infty \sD(\lap^{k/2})$. In general, we will from now on freely use the terminology and results on holomorphic families of pseudodifferential operators which are contained in Appendix \ref{s:holpseope}. Furthermore, we will not distinguish in notation between the selfadjoint unbounded operators
\[
D_1 := \T{cl}(i\frac{d}{dt}) : \sD(D_1) \to L^2(\rr) \q \T{and} \q D_2 : \sD(D_2) \to H_2
\]
and their analogues $\T{cl}(D_1 \ot 1)$ and $\T{cl}(1 \ot D_2)$ on $H$ which are defined as the closures of
\[
D_1 \ot 1 : \sD(D_1) \ot H_2 \to H \q \T{and} \q 1 \ot D_2 : L^2(\rr) \ot \sD(D_2) \to H
\]
respectively.

For each $\ep \in (0,1]$, we define the first order pseudodifferential operator, with respect to $\lap : \sD(\lap) \to H$, $T_1(\ep) : \sD(T_1(\ep)) \to L^2(\rr,H_2)$ as the closure of
\begin{equation}\label{eq:pseone}
T_1(\ep)\big|_{H^\infty} := D_2 h_\ep A + h_\ep A D_2 : H^\infty \to H 
\end{equation}
Furthermore  we define the pseudodifferential operators of order $0$,
\begin{equation}\label{eq:psezer}
T_+(\ep) := h_\ep^2 \cd A^2 + \frac{d h_\ep}{dt} \cd A \q 
T_-(\ep) := h_\ep^2 \cd A^2 - \frac{d h_\ep}{dt} \cd A \q
T_0(\ep) := h_\ep^2 \cd A^2
\end{equation}
all of them with respect to $\lap : \sD(\lap) \to L^2(\rr,H_2)$.
Finally, we let
\[
\al_+(\ep) := T_1(\ep) + T_+(\ep) \q
\al_-(\ep) := T_1(\ep) + T_-(\ep) \q
\al_0(\ep) := T_1(\ep) + T_0(\ep)
\]
denote the associated sums of pseudodifferential operators, which have the common domain $\sD(T_1(\ep)) \su L^2(\rr,H_2)$.

For each $\ep \in (0,1]$, we define the selfadjoint bounded operator
\[
F(\ep) := 2 \frac{d h_\ep}{dt} \cd A : L^2(\rr,H_2) \to L^2(\rr,H_2)
\]
The reason for introducing the above pseudodifferential operators is the following two basic identities:
\begin{equation}\label{eq:basideeps}
\begin{split}
& D_-(\ep) D_+(\ep) = \lap + \al_-(\ep) : \sD(\lap) \to H \q \T{and} \\
& D_+(\ep) D_-(\ep) = \lap + \al_+(\ep) : \sD(\lap) \to H
\end{split}
\end{equation}

\noindent \emph{In this Section we will provide the first application of our rescaling result for the homological index (Theorem \ref{t:adihomind}). To be more precise we shall prove the following theorem, which shows that we may ignore the derivative $\frac{d}{dt}(h_\ep)$ when we compute in the adiabatic limit.}

\begin{thm}\label{t:firappind}
Let $m \in \nn_0$ with $m > p_0/2$ be given. Then there exists a constant $C > 0$ such that,
for all $\la \geq C$,
\[
\T{H-Ind}_\la^m(D_+)
= \lim_{\ep \searrow 0} \T{Tr}\Big( m \cd \la^m \cd F(\ep) \cd \big( \lap + \al_0(\ep) + \la \big)^{-m-1} \Big).
\]

\end{thm}

We start by providing a couple of estimates which will turn out to be useful later on. The reader should here pay particular attention to Lemma \ref{l:limderint}.

We let $\si : \T{OP}^0(\lap) \to \T{OP}^0(\lap)$ denote the algebra automorphism defined by
\[
\si : T \mapsto (\lap + 1) T (\lap + 1)^{-1}
\]

\begin{lemma}\label{l:supderint}
Let $k,l \in \nn_0$ be given. Then
\[
\sup_{\ep \in (0,1]} \big\| \si^k(\frac{d^l h_\ep}{dt})\big\| < \infty
\]
\end{lemma}
\begin{proof}
The proof runs by induction on $k \in \nn_0$.
For $k = 0$ we note that
\begin{equation}\label{eq:derintfct}
\frac{d^l h_\ep}{dt}(t_0) = \ep^l \cd \frac{d^l h}{dt}(\ep \cd t_0)
\end{equation}
for all $t_0 \in \rr$ and all $l \in \nn_0$. Since $\frac{d^l h}{dt}$ is smooth and bounded this proves the statement in this case.

Suppose thus that the statement of the lemma holds for some $k \in \nn_0$. A basic computation then shows that
\begin{equation}\label{eq:indintfct}
\begin{split}
& \si^{k + 1}\big( \frac{d^l h_\ep}{dt} \big) - \si^k\big( \frac{d^l h_\ep}{dt} \big) \\
& \q = - \si^k\big( \frac{d^{l + 2} h_\ep}{dt} \big)(\lap + 1)^{-1} 
- 2 \si^k \big( \frac{d^{l + 1} h_\ep}{dt} \big) \frac{d}{dt} \cd (\lap + 1)^{-1}
\end{split}
\end{equation}
for all $l \in \nn_0$. This computation proves the induction step.
\end{proof}

\begin{lemma}\label{l:limderint}
Let $k \in \nn_0$ and let $l \in \nn$. Then
\[
\lim_{\ep \searrow 0} \big\| \si^k(\frac{d^l h_\ep}{dt}) \big\| = 0
\]
\end{lemma}
\begin{proof}
The proof is again by induction on $k \in \nn_0$.
The induction start follows by the identity in \eqref{eq:derintfct} since $l \geq 1$.
The induction step then follows by the identities in \eqref{eq:indintfct}.
\end{proof}

\begin{lemma}\label{l:supderalp}
Let $k \in \nn_0$. Then
\[
\sup_{\ep \in (0,1]} \big\| \si^k\big( [\lap,\al_+(\ep)](\lap + 1)^{-1}\big)\big\| < \infty
\]
The same result holds when $\al_+(\ep)$ is replaced by $\al_-(\ep)$ or by $\al_0(\ep)$.
\end{lemma}
\begin{proof}
We will focus on the case of $\al_+(\ep)$ since the proof is similar in the two remaining cases.

We start by noticing that
\[
\begin{split}
[\lap, \al_+(\ep)](\lap + 1)^{-1} 
& = \Big( [\lap, h_\ep A] D_2 
+ D_2 [\lap, h_\ep A] \\
& \qq + [\lap, \frac{d h_\ep}{dt} \cd A]
+ [\lap,h_\ep^2 A^2] \Big) (\lap + 1)^{-1}
\end{split}
\]
Since $A \in \T{OP}^0(\lap)$ is independent of $\ep \in (0,1]$ it therefore suffices to show that
\[
\begin{split}
& \sup_{\ep \in (0,1]} \big\| \si^k([\lap,h_\ep] (\lap + 1)^{-1/2}) \big\| < \infty
\, \, , \, \,\,
\sup_{\ep \in (0,1]} \big\| \si^k(h_\ep) \big\| < \infty 
\, \, , \, \,\, \\
& \sup_{\ep \in (0,1]} \big\| \si^k(\frac{d h_\ep}{dt}) \big\| < \infty
\, \, , \, \,\,
\sup_{\ep \in (0,1]} \big\| \si^k\big([ \lap, \frac{d h_\ep}{dt}] (\lap + 1)^{-1/2}\big) \big\| 
< \infty
\end{split}
\]
for each $k \in \nn_0$. But this follows easily from Lemma \ref{l:supderint}.
\end{proof}

With these basic estimates in hand we start estimating resolvents of pseudodifferential operators.

\begin{prop}\label{p:uniresest}
There exists a constant $C > 0$ such that
\[
\sup_{\ep \in (0,1]} \big\| \si^k \big( (\lap + 1) (\lap + \al_+(\ep) + \la)^{-1} \big) \big\| < \infty
\]
for all $k \in \nn_0$ and all $\la \geq C$.

A similar estimate holds when $\al_+(\ep)$ is replaced by $\al_-(\ep)$ or by $\al_0(\ep)$.
\end{prop}
\begin{proof}
We will again focus exclusively on the case of $\al_+(\ep)$ since the proof will carry through with only minor changes in the remaining cases.

Remark first that
\[
\begin{split}
\sup_{\ep \in (0,1]} \big\| \al_+(\ep) (\lap + 1)^{-1/2}\big\|
& \leq \sup_{\ep \in (0,1]} \Big(
\| h_\ep A \cd D_2 (\lap + 1)^{-1/2} \| + \| h_\ep D_2 A (\lap + 1)^{-1/2} \| \Big) \\
& \q + \sup_{\ep \in (0,1]} \Big( 
\| h_\ep^2 A^2 \| + \| \frac{d h_\ep}{dt} \cd A \| 
\Big) \\ 
& < \infty
\end{split}
\]
We may thus choose a constant $C \geq 1$ such that
\[
\sqrt{C}/2 > \sup_{\ep \in (0,1]} \big\| \al_+(\ep) (\lap + 1)^{-1/2}\big\|
\]

Let now $\la \geq C$ be fixed. As in the proof of Proposition \ref{p:resestpow} we then obtain that
\[
(\lap + 1) (\lap + \al_+(\ep) + \la)^{-1} = (\lap + 1)(\lap + \la)^{-1} \cd
(1 + \al_+(\ep)(\lap + \la)^{-1})^{-1}
\]
for all $\ep \in (0,1]$ and all $\la \geq C$. It is therefore sufficient to show that
\begin{equation}\label{eq:sufcon}
\sup_{\ep \in (0,1]} \big\| \si^k \big( (1 + \al_+(\ep)(\lap + \la)^{-1})^{-1} \big) \big\| < \infty
\end{equation}
for all $k \in \nn_0$. This will follow by induction on $k \in \nn_0$.

For $k = 0$ we may compute as follows,
\[
\begin{split}
& \big\| (1 + \al_+(\ep)(\lap + \la)^{-1})^{-1} \big\|
\leq \sum_{j = 0}^\infty \| \al_+(\ep) (\lap + \la)^{-1} \|^j \\
& \q \leq \sum_{j = 0}^\infty \| \al_+(\ep) (\lap + 1)^{-1/2} \|^j \cd \la^{-j/2}
\leq \sum_{j = 0}^\infty (C / \la)^{j/2} \cd (1/2)^j \leq 2
\end{split}
\]
regardless of $\ep \in (0,1]$. This computation thus proves the induction start.

Suppose now that the estimate in \eqref{eq:sufcon} holds for some $k \in \nn_0$. We then compute as follows,
\[
\begin{split}
& \si^{k + 1}\big( (1 + \al_+(\ep)(\lap + \la)^{-1})^{-1} \big) -
\si^k \big( (1 + \al_+(\ep)(\lap + \la)^{-1})^{-1} \big) \\
& \q = \si^k\Big( \big[\lap, \big(1 + \al_+(\ep)(\lap + \la)^{-1}\big)^{-1}\big] (\lap + 1)^{-1}\Big) \\
& \q = - \si^k\Big( \big(1 + \al_+(\ep)(\lap + \la)^{-1}\big)^{-1} \Big)
\cd \si^k\big( [\lap,\al_+(\ep)] (\lap + \la)^{-1}\big) \\
& \qqq \cd \si^k\Big( \big(1 + \al_+(\ep)(\lap + \la)^{-1}\big)^{-1} \Big) \cd (\lap + 1)^{-1}
\end{split}
\]
for all $\ep \in (0,1]$. This computation proves the induction step by an appeal to Lemma \ref{l:supderalp}.
\end{proof}

In the next result we investigate what happens to certain differences of resolvents when computing in the adiabatic limit where the rescaling parameter $\ep \in (0,1]$ tends to zero.

\begin{prop}\label{p:difresvan}
There exists a constant $C > 0$ such that
\[
\lim_{\ep \searrow 0 } \big\| \si^k\big( (\lap + 1)(\lap + \al_+(\ep) + \la)^{-1}
- (\lap + 1)(\lap + \al_0(\ep) + \la)^{-1} \big)
 \big\| = 0
\]
whenever $k \in \nn_0$ and $\la \geq C$. The same result holds with $\al_-(\ep)$ instead of $\al_+(\ep)$.
\end{prop}
\begin{proof}
As usual we will focus on the case of $\al_+(\ep)$ leaving the other case to the reader.
Let us choose the constant $C > 0$ such that Proposition \ref{p:uniresest} holds for both $\al_+(\ep)$ and $\al_0(\ep)$.

Let $\la \geq C$ be fixed. The resolvent identity then yields that
\[
\begin{split}
& (\lap + 1)\big( (\lap + \al_+(\ep) + \la)^{-1} - (\lap + \al_0(\ep) + \la)^{-1}\big) \\
& \q = - (\lap + 1)(\lap + \al_+(\ep) + \la)^{-1} \cd ( T_+(\ep) - T_0(\ep) ) \cd (\lap + \al_0(\ep) + \la)^{-1} \\
& \q = - (\lap + 1)(\lap + \al_+(\ep) + \la)^{-1} \cd \frac{d h_\ep}{dt} \cd A (\lap + \al_0(\ep) + \la)^{-1}
\end{split}
\]
An application of Proposition \ref{p:uniresest} then shows that we can restrict ourselves to proving that
\[
\lim_{\ep \searrow 0} \big\| \si^k\big( \frac{d h_\ep}{dt} \big)\big\| = 0
\]
for all $k \in \nn_0$. But this is a consequence of Lemma \ref{l:limderint}.
\end{proof}

The last result, which we need in order to supply a proof of Theorem \ref{t:firappind}, provides a trace norm estimate on the product of a resolvent and a compactly supported operator valued map.

\begin{prop}\label{p:trauniest}
Let $m  \in \nn$ with $m > p_0/2$ be given. Then
\[
\sup_{\ep \in (0,1] } \big\| F(\ep) \cd (\lap + 1)^{-m-1} \big\|_1 < \infty
\]
\end{prop}
\begin{proof}
As in the proof of Proposition \ref{p:sumpro} we compute that
\[
F(\ep) \cd (\lap + 1)^{-m-1} = 2 A (1 + \De_2)^{-\frac{(m+1)p_0}{p_0+1}} \cd \frac{d h_\ep}{dt} (1 + \De_1)^{-\frac{m+1}{p_0+1}}
\cd G_1^{\frac{m+1}{p_0+1}} G_2^{\frac{(m+1)p_0}{p_0+1}}
\]
for all $\ep \in (0,1]$, where $G_1 = (1 + \De_1)(\lap + 1)^{-1}$ and $G_2 = (1 + \De_2)(\lap + 1)^{-1}$.

It therefore suffices to prove that
\[
\sup_{\ep \in (0,1]}\big\| \frac{dh_\ep}{dt} \cd (1 + \De_1)^{-q}\big\|_1 < \infty
\]
for a fixed real number $q > 1/2$, where $\| \cd \|_1 : \sL^1(L^2(\rr)) \to [0,\infty)$ denotes the trace norm associated with the Hilbert space $L^2(\rr)$.

Let $\ep \in (0,1]$ be fixed and recall from \eqref{eq:conasy} that there exists a $K \in \nn$ such that
\[
\T{supp}(\frac{dh}{dt}) \su [-K,K]
\]
where $\T{supp}( \cd )$ takes the support of a function. It follows that the derivative of the rescaled function $h_\ep : \rr \to \rr$ satisfies
\[
\T{supp}(\frac{d h_\ep}{dt}) \su [-K/\ep,K/\ep]
\]
for all $\ep \in (0,1]$.

It is therefore a consequence of \cite[Theorem 4.5]{Sim:TIA} that there exists a constant $C > 0$, which is independent of $\ep \in (0,1]$, such that
\[
\big\| \frac{d h_\ep}{dt} \cd (1 + \De_1)^{-q} \big\|_1 
\leq C \cd \sup_{t_0 \in \rr}\big|  \frac{dh_\ep}{dt}(t_0) \big| \cd 2 (K/\ep + 1)
\]
To continue, we note that
\[
\sup_{t_0 \in \rr} \big| \frac{dh_\ep}{dt}(t_0) \big| = \ep \cd \sup_{t_0 \in \rr} \big| \frac{dh}{dt}(t_0)\big|
\]
We therefore obtain that
\[
\big\| \frac{d h_\ep}{dt} \cd (1 + \De_1)^{-q} \big\|_1 
\leq C \cd \sup_{t_0 \in \rr}\big|  \frac{dh}{dt}(t_0) \big| \cd 2 (K + \ep)
\]
This estimate proves the lemma.
\end{proof}

We are now ready to prove the main result of this Section. It provides a good first approximation to the homological index and shows how our invariance result in the rescaling parameter $\ep > 0$ plays an important role. We restate the result here for the convenience of the reader:

\begin{thm}
Let $m \in \nn$ with $m > p_0/2$ be given. There exists a constant $C > 0$ such that
\[
\T{H-Ind}_\la^m(D_+)
= \lim_{\ep \searrow 0} \T{Tr}\Big( m \cd \la^m \cd F(\ep) \cd \big( \lap + \al_0(\ep) + \la \big)^{-m-1} \Big)
\]
for all $\la \geq C$.
\end{thm}
\begin{proof}
Let us choose the constant $C > 0$ such that Proposition \ref{p:uniresest} holds for $\al_+(\ep)$, $\al_-(\ep)$ and $\al_0(\ep)$ and furthermore such that Proposition \ref{p:difresvan} holds for $\al_+(\ep)$ and $\al_-(\ep)$.

Let $\la \geq C$.
By the definition of the homological index, by Theorem \ref{t:adihomind}, and by the basic identities in \eqref{eq:basideeps} we have that
\[
\begin{split}
& \T{H-Ind}_\la^m(D_+) \\
& \q = \la^m \cd \lim_{\ep \searrow 0} \T{Tr} \big( (\la + D_-(\ep) D_+(\ep))^{-m} - (\la + D_+(\ep) D_-(\ep))^{-m} \big) \\
& \q = \la^m \cd \lim_{\ep \searrow 0} \Big( \sum_{j = 1}^m \T{Tr}\big( (\lap + \al_-(\ep) + \la)^{-j}  F(\ep)  (\lap + \al_+(\ep) + \la)^{-m+j-1} \big) \Big) \\
& \q = \la^m \cd \lim_{\ep \searrow 0} \Big( \sum_{j = 1}^m 
\T{Tr}\big( F(\ep) \cd (\lap + \al_+(\ep) + \la)^{-m+j-1} \cd (\lap + \al_-(\ep) + \la)^{-j} \big) \Big)
\end{split}
\]

Let thus $j \in \{1,\ldots,m\}$ be fixed. It then suffices to show that
\[
\begin{split}
& 0 = \lim_{\ep \searrow 0}\big\| F(\ep) \cd (\lap + \al_+(\ep) + \la)^{-m+j-1} \cd (\lap + \al_-(\ep) + \la)^{-j} \\
& \qqq - F(\ep) \cd (\lap + \al_0(\ep) + \la)^{-m-1} \big\|_1
\end{split}
\]
But this is now a consequence of Proposition \ref{p:uniresest}, Proposition \ref{p:difresvan}, and Proposition \ref{p:trauniest}. 
\end{proof}

\section{Resolvent expansions of the homological index}\label{s:resexphom}
The starting point for this Section is the formula for the homological index:
\begin{equation}\label{eq:firapp}
\T{H-Ind}_\la^m(D_+) = m \cd \la^m \cd \lim_{\ep \searrow 0} \T{Tr}\big(
F(\ep) \cd \big( \lap + \al_0(\ep) + \la \big)^{-m-1} \big)
\end{equation}
which we obtained in Section \ref{s:firapphom}. For a given $\ep \in (0,1]$, it is now our aim to find an explicit power-series expansion of the integrand
\[
F(\ep) \cd \big( \lap + \al_0(\ep) + \la \big)^{-m-1}
\]
The preliminary form of this expansion will mainly rely on the results in Appendix \ref{s:holpseope} and Appendix \ref{s:tranorhol}. In the next subsection we shall then give the more explicit form of the individual terms in our expansion.

We recall from \eqref{eq:pseone} and \eqref{eq:psezer} that
\[
\al_0(\ep) = T_1(\ep) + T_0(\ep) : \sD( T_1(\ep)) \to H
\]
where $T_0(\ep) := h_\ep^2 A^2$ and $T_1(\ep) : \sD(T_1(\ep)) \to L^2(\rr,H_2)$ is the closure of the anti-commutator
\[
D_2 h_\ep A + h_\ep A D_2 : H^\infty \to H
\]

In order to compute the integrand in \eqref{eq:firapp} we now introduce an extra parameter. Thus, for each $z \in \cc$ and each $\ep \in (0,1]$ we define
\[
\al_0(\ep,z) := z \cd T_1(\ep) + z^2 \cd T_0(\ep) : \sD(T_1(\ep)) \to L^2(\rr,H_2)
\]
and it follows immediately that $\al_0(\ep, \cd) : \cc \to \T{OP}^1(\lap)$ defines a holomorphic family in the sense of Definition \ref{d:holfampse} (the compact Hausdorff space $X$ is here just a point). Furthermore, since $m > p_0/2$ we know that
\[
F(\ep) \cd (\lap + 1)^{-m-1} \in \sL^1( L^2(\rr,H_2)) 
\]
Finally, we clearly have that
\[
\sup_{\ep \in (0,1] \, , \, \, z \in \B B_2(0)} \| \al_0(\ep,z) (\lap + 1)^{-1/2} \| < \infty
\]

We may thus apply Proposition \ref{p:tranorpow} to obtain the following result:

\begin{prop}\label{p:resexphom}
Let $m \in \nn$ with $m > p_0/2$ be given and let $\la \geq 1$ be a constant with
\[
\sqrt{\la} > \sup_{\ep \in (0,1] \, , \, \, z \in \B B_2(0)} \| \al_0(\ep,z) (\lap + 1)^{-1/2} \|
\]
Then, for each $\ep \in (0,1]$, we have that
\[
F(\ep) \cd \big(\lap + \al_0(\ep) + \la \big)^{-m-1} 
= \sum_{l = 0}^\infty \frac{1}{l!} F(\ep) \cd \frac{d^l}{dz}\big( \lap + \al_0(\ep, \cd ) + \la \big)^{-m-1} \big|_{z = 0}
\]
where the sum converges absolutely in trace norm.
\end{prop}

In the next subsection we will derive a more tangible expression for the derivatives
\[
\frac{d^l}{dz}\big( \lap + \al_0(\ep, \cd )  + \la \big)^{-m-1} \big|_{z = 0}
\]
which appear in our resolvent expansion. We will derive this expression in a more abstract context since the more general result may be useful elsewhere.

\subsection{Higher derivatives of resolvents}\label{s:higderres}
Throughout this Section we let $H$ be a separable Hilbert space and $\De : \sD(\De) \to H$ be a positive unbounded operator. On top of this data, we will consider two pseudodifferential operators $T_1 \in \T{OP}^1(\De)$ and  $T_0 \in \T{OP}^0(\De).$ Finally, we let $U \su \cc$ be a bounded open neighborhood of $0 \in \cc$.

We define the holomorphic map
\begin{equation}\label{eq:alphol}
\al : U \to \T{OP}^1(\De) \q \al(z) := z T_1 + z^2 T_0 : \sD(T_1) \to H
\end{equation}
and we fix a $\la > 0$ such that the resolvent
\[
R_\la := (\De + \al + \la )^{-1} : U \to \T{OP}^{-2}(\De)
\]
is well-defined and holomorphic, see Lemma \ref{l:invhol}.

Let us also define the derivatives
\[
\frac{d \al}{dz}\big|_z := T_1 + 2 z T_0 \q \M{and} \q \frac{d^2 \al}{dz} \big|_z := 2 T_0
\]
for all $z \in U$. All the higher derivatives are trivial.

\noindent \emph{In this Section we are interested in computing the higher derivatives of an arbitrary power of the resolvent, thus the holomorphic map
\[
\frac{d^l}{dz} R_\la^{m + 1} : U \to \sL(H)
\]
for each $m \in \nn_0$ and each $l \in \nn$.}

For each $j \in \nn$ we let $\cc[\nn^j]$ denote the vector space with a standard basis vector $\de_K$ for each $K \in \nn^j$. For $j = 0$ we define $\cc[\nn^0] := \cc \de_\emptyset := \cc$.

We start with a convenient definition:

\begin{dfn}\label{d:condefone}
For each $j \in \nn$ and each $K \in \nn^j$ we define the holomorphic operator valued map
\[
\begin{split}
\inn{R_\la^{m+1}; \de_K} & := \sum_{M \in \nn^{j+1}, \, |M| = m+ 1 + j} (-1)^j \cd R_\la^{m_1} \cd \frac{d^{k_1} \al}{dz} R_\la^{m_2} \clc \frac{d^{k_j} \al}{dz} R_\la^{m_{j+1}} \\
& \q : U \to \sL(H)
\end{split}
\]
For $j = 0$, define $\inn{R_\la^{m+1}; \de_\emptyset} := R_\la^{m+1}$.
By linearity we obtain a holomorphic operator valued map
$
\inn{R_\la^{m+1}, \Te} : U \to \sL(H)
$
for each $j \in \nn_0$ and each $\Te \in \cc[\nn^j]$.
\end{dfn}

Let us now compute the first derivative of the map defined above. We introduce the operations
\begin{equation}\label{eq:comideder}
\begin{split}
& s : \cc[\nn^j] \to \cc[\nn^{j+1}] \q s : \de_K \mapsto \sum_{i = 1}^{j + 1} \de_{(k_1,\ldots,1,k_i,\ldots,k_j)} \q \T{and} \\
& e : \cc[\nn^j] \to \cc[\nn^j] \q e : \de_K \mapsto \sum_{i = 1}^j \de_{(k_1,\ldots,k_i + 1,\ldots,k_j)}
\end{split}
\end{equation}
In particular we note that $s(\de_{\emptyset}) := \de_1$ and $e(\de_{\emptyset}) := 0$.

\begin{lemma}\label{l:derfunfor}
Let $j \in \nn_0$ and $\Te \in \cc[\nn^j]$ be given. Then
\[
\frac{d}{dz} \inn{R_\la^{m+1};\Te} = \inn{R_\la^{m+1};(s + e)(\Te) }
\]
\end{lemma}
\begin{proof}
Suppose first that $j = 0$. Then
\[
\begin{split}
\frac{d}{dz} R_\la^{m+1}
& = \sum_{k = 0}^m R_\la^k \cd \frac{d R_\la}{dz} \cd R_\la^{m - k} \\
& = - \sum_{k = 0}^m R_\la^{k+1} \cd \frac{d \al}{dz} \cd R_\la^{m-k+1}
= \inn{R_\la^{m+1};\de_1}
\end{split}
\]

For general $j \in \nn$ and $K \in \cc[\nn^j]$, the Leibniz rule then yields the desired formula
\[
\begin{split}
\frac{d}{dz}\inn{R_\la^{m+1};\de_K} & = \sum_{i = 1}^{j+1} \inn{R_\la^{m+1}; \de_{(k_1,\ldots,1,k_i,\ldots,k_j)} } \\
& \qq + \sum_{i = 1}^j \inn{R_\la^{m+1};\de_{(k_1,\ldots,k_i + 1,\ldots,k_j)}}
\end{split}.
\]
\end{proof}

The desired explicit formula for the higher derivatives of powers of resolvents now follows immediately by Lemma \ref{l:derfunfor} and an induction argument.

\begin{prop}\label{p:higderres}
Let $l, m \in \nn_0$ be given. Then $\frac{d^l}{dz} R_\la^{m+1} = \inn{R_\la^{m+1}; (s + e)^l(\de_\emptyset)}$.
\end{prop}

\section{Isolation of the interpolating function}\label{s:isointfct}
In this Section we will continue our investigation of the derivatives appearing in the resolvent expansion of Proposition \ref{p:resexphom}. We are thus interested in the bounded operators of the form
\begin{equation}\label{eq:deralpzer}
\frac{d^l}{dz} \big( \lap + \al_0(\ep,\cd) + \la \big)^{-m-1} \big|_{z = 0} \q l,m \in \nn_0
\end{equation}
where we recall that the holomorphic map
\[
\al_\ep := \al_0(\ep,\cd) : \cc \to \T{OP}^1(\lap)
\]
is defined by putting $\al_\ep(z)$ equal to the closure of
\[
z \cd (D_2 h_\ep A + A h_\ep D_2) + z^2 \cd h_\ep^2 A^2 : H^\infty \to H
\]
for all $z \in \cc$ and all $\ep \in (0,1]$. More precisely we will compare the derivative in \eqref{eq:deralpzer} to the derivative
\[
h_\ep^l \cd \frac{d^l}{dz}\big( \lap + \be + \la \big)^{-m-1} \big|_{z = 0}
\]
where the holomorphic map 
\[
\be : \cc \to \T{OP}^1(\lap)
\]
is defined by letting $\be(z)$ equal the closure of
\[
z \cd (D_2 A + A D_2 ) + z^2 \cd A^2 : H^\infty \to H
\]
for all $z \in \cc$. This comparison will take place in the adiabatic limit where $\ep$ tends to zero from above and will rely on the explicit formula for derivatives of powers of resolvents obtained in Proposition \ref{p:higderres}. The result of this Section can thus be interpreted as a noncommutative analogue of the chain rule in calculus. Indeed, instead of a real parameter $\mu \in \rr$ we have rescaled the map $\be$ by the real valued function $h_\ep$ and we are then investigating the higher derivatives.

Let us once and for all fix a constant $\la_0 \geq 1$ such that the resolvents
\[
\begin{split}
& R(\ep) : \B B_2(0) \to \T{OP}^{-2}(\lap) \q R(\ep) : z \mapsto (\lap + \al_\ep(z) + \la_0)^{-1} \q \M{and} \\
& S : \B B_2(0) \to \T{OP}^{-2}(\lap) \q S : z \mapsto (\lap + \be(z) + \la_0)^{-1}
\end{split}
\]
are well-defined and holomorphic for all $\ep \in (0,1]$. We remark that even though $R(\ep)$ and $S$ depends on $\la_0$ we have surpressed this dependence from the notation.

The main goal of this Section is then to prove the following theorem:

\begin{thm}\label{t:adicharul}
Let $l \in \nn_0$ and $m \in \nn_0$ be given. We then have the convergence result:
\[
\lim_{\ep \searrow 0}
\big\| (\lap + 1)^{m+1} \big( \frac{d^l}{dz} R^{m + 1}(\ep) \big|_{z = 0} - h_\ep^l \cd \frac{d^l}{dz} S^{m + 1} \big|_{z = 0} \big) \big\| = 0
\]
\end{thm}

In words the content of the above theorem is that it is possible to isolate the interpolating function $h_\ep$ in the adiabatic limit. More precisely, we observe that in the expression for the derivative $\frac{d^l}{dz}R(\ep)\big|_{z = 0}$, the interpolating function $h_\ep$ appears in a complicated combination with resolvents of $\lap$ and the operator $A$. On the contrary, in the expression $h_\ep^l \cd \frac{d^l}{dz} S^{m + 1} \big|_{z = 0}$, the interpolating function $h_\ep$ only appears at the very front. The result of Theorem \ref{t:adicharul} therefore provides a key step in our computation of the homological index.
\medskip

Let us begin the proof of the above theorem by making a basic observation:

\begin{lemma}\label{l:derintfro}
\[
\frac{d\al_\ep}{dz}(0) = h_\ep \cd \frac{d\be}{dz}(0) \q \M{and} \q
\frac{d^2 \al_\ep}{dz}(0) = h_\ep^2 \cd \frac{d^2\be}{dz}(0)
\]
\end{lemma}

We continue by introducing some convenient notation:

\begin{dfn}\label{d:condeftwo}
Let $m \in \nn_0$. For each $j \in \nn$ and $y_1,\ldots,y_j \in \T{OP}^2(\lap)$ we define the pseudodifferential operator of order $-2m - 2$, 
\[
\inn{y_1,\ldots,y_j} := (-1)^j \sum_{M \in \nn^{j+1}, \, |M| = m + j + 1} (\lap + \la_0)^{-m_1} y_1 
\clc (\lap + \la_0)^{-m_j} y_j \cd (\lap + \la_0)^{-m_{j+1}}
\]
For $j = 0$, we put $\inn{y_1,\ldots,y_j} := (\lap + \la_0)^{-m-1}$.
\end{dfn}

The relation between the above notation and the notation introduced in Definition \ref{d:condefone} is given by the following observation:
\[
\begin{split}
& \binn{ \frac{d^{k_1}\be}{dz}(0), \ldots, \frac{d^{k_j} \be}{dz}(0) } = \inn{S^{m + 1} ; \de_K}(0) \q \T{and} \\
& \binn{ \frac{d^{k_1} \al_\ep}{dz}(0), \ldots, \frac{d^{k_j} \al_\ep}{dz}(0) } = \inn{R^{m + 1}(\ep); \de_K}(0)
\end{split}
\]
for all $j \in \nn_0$ and $K = (k_1,\ldots,k_j) \in \nn^j$. We now compare the pseudodifferential operators
\[
\inn{R^{m + 1}(\ep) ; \de_K}(0) \q \M{and} \q
h_\ep^{|K|} \cd \inn{S^{m + 1}; \de_K }(0)
\]
where $|K| := k_1 \plp k_j$:

\begin{lemma}\label{l:intfctfro}
Let $m \in \nn_0$, $j \in \nn$, $K \in \nn^j$ and $\ep \in (0,1]$ be given. Then
\[
\begin{split}
& \inn{ R^{m + 1}(\ep) ; \de_K }(0) - h_\ep^{|K|} \cd \inn{ S^{m+1}; \de_K }(0) \\
& \q = \sum_{p = 1}^j \binn{\frac{d^{k_1} \al_\ep}{dz}(0), \ldots, \frac{d^{k_{p-1}} \al_\ep}{dz}(0),
[\lap, h_\ep^{k_p \plp k_j}], \frac{d^{k_p} \be}{dz}(0), \ldots,\frac{d^{k_j} \be}{dz}(0) }
\end{split}
\]
\end{lemma}
\begin{proof}
We will prove that
\begin{equation}\label{eq:indfordow}
\begin{split}
& \binn{ \frac{d^{k_1}\al_\ep}{dz}(0), \ldots, \frac{d^{k_j} \al_\ep}{dz}(0) } \\
& \q = \binn{\frac{d^{k_1} \al_\ep}{dz}(0), \ldots, \frac{d^{k_{l-2}} \al_\ep}{dz}(0), 
h_\ep^{k_{l-1} \plp k_j} \frac{d^{k_{l-1}} \be}{dz}(0), \frac{d^{k_l} \be}{dz}(0), \ldots,\frac{d^{k_j} \be}{dz}(0) } \\
& \qq + \sum_{p = l}^j \binn{\frac{d^{k_1} \al_\ep}{dz}(0), \ldots, \frac{d^{k_{p-1}} \al_\ep}{dz}(0),
[\lap, h_\ep^{k_p \plp k_j}], \frac{d^{k_p} \be}{dz}(0), \ldots,\frac{d^{k_j} \be}{dz}(0) }
\end{split}
\end{equation}
for all $l \in \{1,\ldots,j+1\}$.

The proof runs by induction on $l \in \{1,\ldots,j+1\}$, starting with $l = j+1$ and moving downwards.

For $l = j+1$, the formula reads
\[
\binn{ \frac{d^{k_1}\al_\ep}{dz}(0), \ldots, \frac{d^{k_j} \al_\ep}{dz}(0) }
= \binn{ \frac{d^{k_1}\al_\ep}{dz}(0), \ldots, h_\ep^{k_j} \frac{d^{k_j} \be}{dz}(0) } \ ,
\]
but this identity is a consequence of Lemma \ref{l:derintfro}.

Suppose now that the formula in \eqref{eq:indfordow} holds for some $l \in \{2,\ldots,j+1\}$. It is then enough to show that
\begin{equation}\label{eq:indsteres}
\begin{split}
& \binn{\frac{d^{k_1} \al_\ep}{dz}(0), \ldots, \frac{d^{k_{l-2}} \al_\ep}{dz}(0), 
h_\ep^{k_{l-1} \plp k_j} \frac{d^{k_{l-1}} \be}{dz}(0), \frac{d^{k_l} \be}{dz}(0), \ldots,\frac{d^{k_j} \be}{dz}(0) } \\
& \q = \binn{\frac{d^{k_1} \al_\ep}{dz}(0), \ldots, \frac{d^{k_{l-3}} \al_\ep}{dz}(0), 
h_\ep^{k_{l-2} \plp k_j} \frac{d^{k_{l-2}} \be}{dz}(0), \frac{d^{k_{l-1}} \be}{dz}(0), \ldots,\frac{d^{k_j} \be}{dz}(0) } \\
& \qq + \binn{\frac{d^{k_1} \al_\ep}{dz}(0), \ldots, \frac{d^{k_{l-2}} \al_\ep}{dz}(0), [\lap,h_\ep^{k_{l-1} \plp k_j} ],
\frac{d^{k_{l-1}} \be}{dz}(0), \ldots,\frac{d^{k_j} \be}{dz}(0) }
\end{split}
\end{equation}

To this end, we first notice that
\[
h_\ep^{k_{l-2} \plp k_j} \frac{d^{k_{l-2}} \be}{dz}(0) (\xi) =
\frac{d^{k_{l-2}} \al_\ep}{dz}(0) h_\ep^{k_{l-1} \plp k_j} (\xi)
\]
for all $\xi \in H^\infty$. Indeed, this follows since $\be(z) h_\ep(\xi) = h_\ep \be(z) (\xi)$ for all $z \in \cc$.

We then let $m_{l-1} \in \nn$ be given and compute as follows:
\[
\begin{split}
& (\lap + \la_0)^{-m_{l-1}} h_\ep^{k_{l-1} \plp k_j}
= h_\ep^{k_{l-1} \plp k_j} (\lap + \la_0)^{-m_{l-1}} + \big[ (\lap + \la_0)^{-m_{l-1}}, h_\ep^{k_{l-1}\plp k_j}\big]  \\
& \q = h_\ep^{k_{l-1} \plp k_j} (\lap + \la_0)^{-m_{l-1}} \\
& \qqq - \sum_{n = 1}^{m_{l-1}} (\lap + \la_0)^{-n} [\lap,h_\ep^{k_{l-1} \plp k_j}] (\lap + \la_0)^{-(m_{l-1} + 1 - n)}
\end{split}
\]

These two computations clearly imply the identity in \eqref{eq:indsteres}, and the lemma is therefore proved.
\end{proof}

It is now our intention to estimate the size of the right-hand side of the identity in Lemma \ref{l:intfctfro}. We therefore analyze the general size of the pseudodifferential operators appearing in Definition \ref{d:condeftwo}.
Recall the introduction of the automorphism $\sigma : \T{OP}^0(\lap) \to \T{OP}^0(\lap)$ after Theorem \ref{t:firappind}.

\begin{lemma}\label{l:funest}
Let $m \in \nn_0$, $j \in \nn$, and let $y_1,\ldots,y_j \in \T{OP}^2(\lap)$ be given. Then we have the norm estimate
\[
\begin{split}
& \| (\lap + 1)^{m+1} \inn{y_1,\ldots,y_j} \| \\ 
& \q \leq {j + m \choose j}  \cd \sup_{I \in \nn^j\, , \, 1 \leq i_j \leq \ldots \leq i_1 \leq m + 1} \Big( 
\| \si^{i_1}( (\lap + \la_0)^{-1} y_1 )  \| \\ 
& \qqq \qqq \qqq \clc \| \si^{i_j}( (\lap + \la_0)^{-1} y_j)  \| \Big)
\end{split}
\]
\end{lemma}
\begin{proof}
Define the bounded operator $F := (\lap + 1)(\lap + \la_0)^{-1}$ and notice that $\| F \| \leq 1$ since $\la_0 \geq 1$ by definition.

Let $\xi \in H^\infty$. We may compute as follows:
\[
\begin{split}
& (\lap + 1)^{m+1} \inn{y_1,\ldots,y_j}(\xi) \\
& \q = \sum_{M \in \nn^{j+1}, \, |M| = m + j + 1} (-1)^j \cd (\lap + 1)^{m+1} (\lap + \la_0)^{-m_1} y_1  \\
& \qqq \qqq \qqq \clc (\lap + \la_0)^{-m_j} y_j \cd (\lap + \la_0)^{-m_{j+1}}(\xi) \\
& \q = \sum_{M \in \nn^{j+1}, \, |M| = m + j + 1} (-1)^j \cd F^{m_1 -1} \si^{m+2 - m_1}( (\lap + \la_0)^{-1} y_1)  \\
& \qqq \qqq \q \clc F^{m_j - 1} \si^{m + j + 1 - m_1 - \ldots - m_j}( (\lap + \la_0)^{-1} y_j)  \\ 
& \qqq \qqq \qqq \cd F^{m_{j+1}} (\xi)
\end{split}
\]

We can thus estimate the operator norm in the following way:
\[
\begin{split}
& \|(\lap + 1)^{m+1} \inn{y_1,\ldots,y_j}\| \\ 
& \q \leq 
\sum_{M \in \nn^{j+1}, \, |M| = m + j + 1} \| \si^{m+2 - m_1}( (\lap + \la_0)^{-1} y_1)  \| \\
& \qqq \qqq  \q \clc \| \si^{m + j + 1 - m_1 - \ldots - m_j}( (\lap + \la_0)^{-1} y_j) \| \\
& \q \leq \sum_{M \in \nn^{j+1}, \, |M| = m + j + 1} \sup_{I \in \nn^j , \, 1 \leq i_j \leq \ldots \leq i_1 \leq m + 1} 
\Big( \| \si^{i_1}( (\lap + \la_0)^{-1} y_1)  \| \\
& \qqq \qqq \qqq \q \clc \| \si^{i_j}( (\lap + \la_0)^{-1} y_j)  \| \Big) \\
& \q = {j + m \choose j} \sup_{I \in \nn^j , \, 1 \leq i_j \leq \ldots \leq i_1 \leq m + 1}
\| \si^{i_1}( (\lap + \la_0)^{-1} y_1)  \|
\clc \| \si^{i_j}( (\lap + \la_0)^{-1} y_j)  \|
\end{split}
\]
This proves the lemma.
\end{proof}

For each $\ep \in (0,1]$, we define
\[
\Te(\ep) := \sup_{i \in \{1,\ldots,m+1\} } \big\| \si^i\big( (\lap + 1)^{-1} [\lap,h_\ep] \big) \big\|
\]

\begin{lemma}\label{l:estforter}
Let $m \in \nn_0$. There exists a constant $C > 0$, which is independent of $\la_0 \geq 1$, such that the estimate
\[
\begin{split}
& \big\| (\lap + 1)^{m+1} \big( \inn{R^{m+1}(\ep) ; \de_K }(0) - h_\ep^{|K|} \cd \inn{S^{m+1} ; \de_K}(0) \big) \big\| \\
& \q \leq \la_0^{-|K|/2} \cd C^{|K|} \cd \Te(\ep)
\end{split}
\]
holds for all $j \in \nn$, all $K \in \{1,2\}^j$ and all $\ep \in (0,1]$.
\end{lemma}
\begin{proof}
By Lemma \ref{l:supderint} we may choose a constant $C_0 \geq 1$ such that
\[
\| \si^i(h_\ep) \| \leq C_0 \, \, , \, \, \,
\big\| (\lap + 1)^{-1/2} \si^i(\frac{d\be}{dz}(0))\big\| \leq C_0 \, \, , \, \, \,
\big\| \si^i(\frac{d^2\be}{dz}(0))\big\| \leq C_0
\]
for all $i \in \{0,1,\ldots,m+1\}$ and all $\ep \in (0,1]$. Let then $j \in \nn$, $K \in \{1,2\}^j$, and $\ep \in (0,1]$ be given.
By Lemma \ref{l:intfctfro}, it is enough to estimate the quantity
\[
\begin{split}
& \sum_{p = 1}^j \Big\| (\lap + 1)^{m+1} \cd 
\binn{\frac{d^{k_1} \al_\ep}{dz}(0), \ldots, \frac{d^{k_{p-1}} \al_\ep}{dz}(0), 
[\lap,h_\ep^{k_p \plp k_j}], 
\frac{d^{k_p} \be}{dz}(0),\clc \frac{d^{k_j} \be}{dz}(0)  } \Big\|
\end{split}
\]

Next,
\[
\begin{split}
& \big\| \si^i\big( (\lap + 1)^{-1} [\lap, h_\ep^{k_p \plp k_j}]\big) \big\| \\
& \q \leq \sum_{r = 0}^{k_p \plp k_j - 1} \| \si^{i-1}(h_\ep) \|^r \cd \big\| \si^i\big( (\lap + 1)^{-1} [\lap,h_\ep] \big) \big\| \cd \| \si^i(h_\ep)\|^{k_p \plp k_j - 1 - r} \\
& \q \leq \Te(\ep) \cd C_0^{k_p \plp k_j - 1} \cd |K|
\end{split}
\]
for all $i \in \{1,\ldots,m+1\}$. Furthermore, since $\frac{d^k \be}{dz}$ is a pseudodifferential operator of order $2 - k$, we have that
\[
\begin{split}
\big\| \si^i \big(  (\lap + \la_0)^{-1} \frac{d^k \al_\ep}{dz}(0) \big) \big\|
= \big\| \si^i \big( (\lap + \la_0)^{-1} \frac{d^k \be}{dz}(0) h_\ep^k \big) \big\|
& \leq \la_0^{-k/2} \cd C_0^{1 + k} \q \T{and} \\
\big\| \si^i \big( (\lap + \la_0)^{-1} \frac{d^k \be}{dz}(0) \big) \big\| & \leq \la_0^{-k/2} \cd C_0
\end{split}
\]
for all $k \in \{1,2\}$, all $i \in \{1,\ldots,m+1 \}$. An application of Lemma \ref{l:funest} therefore yields that
\begin{equation}\label{eq:estforter}
\begin{split}
& \sum_{p = 1}^j \Big\| (\lap + 1)^{m+1} \cd 
\binn{\frac{d^{k_1} \al_\ep}{dz}(0), \ldots, \frac{d^{k_{p-1}} \al_\ep}{dz}(0), 
[\lap,h_\ep^{k_p \plp k_j}], \frac{d^{k_p} \be}{dz}(0),\clc \frac{d^{k_j} \be}{dz}(0)  } \Big\| \\
& \q \leq \sum_{p = 1}^j {j + 1 + m \choose j + 1} \cd \la_0^{-|K|/2} \cd C_0^{j + |K| - 1} \cd \Te(\ep) \cd |K| \\
& \q \leq |K|^2 \cd 2^{|K| + 1 + m} \cd \la_0^{-|K|/2} \cd C_0^{2 |K|} \cd \Te(\ep)
\end{split}
\end{equation}

The estimates in \eqref{eq:estforter} prove the present lemma.
\end{proof}

\begin{remark}
With more care it is possible to obtain a better estimate of the left hand side of Lemma \ref{l:estforter} in so far that $\la_0^{-|K|/2}$ can be replaced by $\la_0^{-(|K| + 1)/2}$. This can be done at the expense of replacing $(\lap + 1)^{-1}$ in the expression for $\Te(\ep)$ by $(\lap + 1)^{-1/2}$. We do not present this estimate here, since it is irrelevant for the developments of the present text. 
\end{remark}

\begin{prop}\label{p:higderest}
Let $m \in \nn_0$. There exists a constant $C > 0$, which is independent of $\la_0 \geq 1$, such that the estimate
\[
\Big\| (\lap + 1)^{m+1} \frac{d^l}{dz} \big( R^{m + 1}(\ep) - h_\ep^l S^{m + 1} \big)(0)\Big\|
\leq \Te(\ep) \cd l! \cd C^l \cd \la_0^{-l/2}
\]
holds for all $l \in \nn_0$ and all $\ep \in (0,1]$.
\end{prop}
\begin{proof}
The inequality is true for $l = 0$ since $\be(0) = \al_\ep(0) = 0$ for all $\ep \in (0,1]$.

By Proposition \ref{p:higderres} we have that
\[
\begin{split}
& (\lap + 1)^{m+1} \frac{d^l}{dz} \big( R^{m+1}(\ep) - h_\ep^l S^{m+1}  \big)(0) \\
& \q = (\lap + 1)^{m+1} \big( \inn{R^{m+1}(\ep);(s + e)^l(\de_{\emptyset})}(0) - h_\ep^l \cd \inn{S^{m+1}; (s + e)^l(\de_{\emptyset})}(0)  \big)
\end{split}
\]
for all $l \in \nn$ and all $\ep \in (0,1]$. Recall here that the maps $s$ and $e$ are defined in \eqref{eq:comideder}.

For each $l \in \nn$ we now notice that $(s + e)^l(\de_\emptyset)$ can be written as a finite sum of basis vectors $\de_K$, where each subset $K$ has length $|K| = l$. Furthermore, there are clearly less than $2^l \cd l!$ elements in this sum. Here each $\de_K$ is chosen to have the coefficient $1$ even though there may of course be many repetitions in the sum we are looking at.  

By Lemma \ref{l:estforter} and the above observations, we may thus choose a constant $C_0 > 0$, which is independent of $\la_0 \geq 1$, such that
\[
\begin{split}
\big\| (\lap + 1)^{m+1} \frac{d^l}{dz} \big( R^{m+1}(\ep) - h_\ep^l S^{m + 1}\big)(0) \big\|
\leq \Te(\ep) \cd \la_0^{-l/2} \cd C_0^l \cd l! \cd 2^l
\end{split}
\]
for all $l \in \nn$ and all $\ep \in (0,1]$.

This estimate proves the present proposition.
\end{proof}

We emphasize that the constant $C > 0$ depends on the fixed nonnegative integer $m \in \nn_0$. It is however independent of both $\la_0 \geq 1$, $l \in \nn_0$, and $\ep \in (0,1]$.

The result of Lemma \ref{l:limderint} implies that
\[
\Te(\ep) = \sup_{i \in \{1,\ldots,m+1\} } \big\| \si^i\big( (\lap + 1)^{-1} [\lap,h_\ep] \big) \big\|
\]
becomes zero in the limit where $\ep$ tends to zero from above. The main Theorem \ref{t:adicharul} announced in the beginning of this Section is therefore a consequence of the above Proposition \ref{p:higderest}.

\section{Second approximation of the homological index}\label{s:secapphom}
In this Section we will recollect what we have obtained thus far. Let us therefore recall that $A : H_2 \to H_2$ is a pseudodifferential operator of order $0$ with respect to $\De_2 := D_2^2 : \sD(\De_2) \to H_2$ and that $h : \rr \to \rr$ is a smooth bounded function such that $\frac{dh}{dt} : \rr \to \rr$ has compact support. Recall furthermore that $A$ satisfies the relative Schatten class conditions described in Assumption \ref{a:psesum} with respect to $\De_2$ and a fixed number $p_0 \in (0,\infty)$. We let $D_1 : \sD(D_1) \to L^2(\rr)$ denote the closure of the unbounded operator $i \frac{d}{dt} : C_c^\infty(\rr) \to L^2(\rr)$. {\bla As usual we will also apply the notation $D_1$ and $D_2$ for the closures of
\[
D_1 \ot 1 : \sD(D_1) \ot H_2 \to L^2(\rr,H_2) \q \T{and} \q 1 \ot D_2 : L^2(\rr) \ot \sD(D_2) \to L^2(\rr,H_2)
\]}

Let $m \in \nn$ with $m > p_0/2$ be given. The main theorem of this Section expresses the homological index in degree $m$ of the closed unbounded operator
\[
D_+ := -i D_1 + D_2 + h A : \sD(D_1) \cap \sD(D_2) \to L^2(\rr,H_2)
\]
as a power-series expansion. Each term in this expansion is a product where the first term only contains a multiplication operator in the flow parameter and the second term only contains operators on the `inner' Hilbert space $H_2$ and differentiation with respect to the flow parameter (there are no multiplication operators in the flow parameter in the second term). This will make it possible  to compute the trace in the flow direction which will take the form of an integral over the real line. This computation will be carried out in the next Section.

\begin{thm}\label{t:secapphom}
Let $m \in \nn$ with $m > p_0/2$ be given. Then there exists a constant $C > 0$ such that
\[
\T{H-Ind}_\la^m(D_+) = m \cd \la^m \cd \lim_{\ep \searrow 0} \sum_{l = 0}^\infty \frac{1}{l!} \cd \T{Tr}\big( F(\ep) \cd h_\ep^l \cd
\frac{d^l}{dz}( \lap + \be + \la)^{-m-1} \big|_{z=0} \big) 
\]
for all $\la \geq C$ and where the sum converges absolutely for all $\ep \in (0,1]$.
\end{thm}
\begin{proof}
By Theorem \ref{t:firappind} there exists a constant $C_0 > 0$ such that
\[
\T{H-Ind}_\la^m(D_+) = m \cd \la^m \cd \lim_{\ep \searrow 0} 
\T{Tr}\big( F(\ep) \cd (\lap + \al_0(\ep) + \la)^{-m-1} \big)
\]
for all $\la \geq C_0$.

By Proposition \ref{p:resexphom} there exists a constant $C_1 > 0$ such that
\[
\begin{split}
& m \cd \la^m \cd \T{Tr}\big( F(\ep) \cd ( \lap + \al_0(\ep) + \la )^{-m-1} \big) \\
& \q = m \cd \la^m \cd \sum_{l = 0}^\infty \frac{1}{l!} \cd \T{Tr}\big(
F(\ep) \cd \frac{d^l}{dz} (\lap + \al_0(\ep,\cd) + \la )^{-m-1} \big|_{z=0} 
\big)
\end{split}
\]
where the sum converges absolutely for all $\ep \in (0,1]$ and all $\la \geq C_1$.

By Proposition \ref{p:higderest} there exists a constant $C_2 > 0$ such that
\begin{equation}\label{eq:higderestI}
\begin{split}
& \big\| (\lap + 1)^{m+1} \cd \frac{d^l}{dz}\big(  (\lap + \al_0(\ep,\cd) + \la)^{-m-1} 
- h_\ep^l \cd (\lap + \be + \la)^{-m-1} \big)\big|_{z =0} \big\| \\
& \q \leq \Te(\ep) \cd l! \cd (1/2)^l
\end{split}
\end{equation}
for all $\la \geq C_2$, all $l \in \nn_0$ and all $\ep \in (0,1]$, where
\[
\Te(\ep) = \sup_{i \in \{1,\ldots,m+1\} } \| \si^i( (\lap + 1)^{-1}[\lap,h_\ep]) \| 
\]

To ease the notation, define the numbers
\[
\begin{split}
a_l(\la,\ep) & := \T{Tr}\big(
F(\ep) \cd \frac{d^l}{dz}(\lap + \al_0(\ep,\cd) + \la)^{-m-1}\big|_{z = 0} \big) \q \T{and} \\
b_l(\la,\ep) & := \T{Tr}\big( F(\ep) \cd h_\ep^l \cd \frac{d^l}{dz} (\lap + \be + \la)^{-m-1}\big|_{z = 0} \big)
\end{split}
\]
for all $\la \geq C_2$ and all $\ep \in (0,1]$.

It then follows by the estimate in \eqref{eq:higderestI} that the sum $m \cd \la^m \sum_{l = 0}^\infty \frac{1}{l!} \cd (a_l(\la,\ep) - b_l(\la,\ep))$ converges absolutely, and furthermore that
\[
m \cd \la^m \cd \sum_{l = 0}^\infty \frac{1}{l!} \cd | a_l(\la,\ep) - b_l(\la,\ep)|
\leq m \cd \la^m \cd 2 \cd \| F(\ep)(\lap + 1)^{-m-1} \|_1 \cd \Te(\ep)
\]
for all $\la \geq C_2$ and all $\ep \in (0,1]$.

From Lemma \ref{l:limderint} and Proposition \ref{p:trauniest} we know that
\[
\lim_{\ep \searrow 0} \big( m \cd \la^m \cd 2 \cd \| F(\ep)(\lap + 1)^{-m-1} \|_1 \cd \Te(\ep)  \big) = 0
\]
Therefore, we may conclude that
\[
m \cd \la^m \cd \lim_{\ep \searrow 0}  \Big( \sum_{l = 0}^\infty \frac{1}{l!} \big|a_l(\la,\ep) - b_l(\la,\ep)\big| \Big) = 0
\]
for all $\la \geq C_2$.

Put now $C := \T{max}\{ C_0,C_1,C_2 \}$ and let $\la \geq C$ be given. It then clearly follows by the above computations that
\[
\begin{split}
\T{H-Ind}_\la^m(D_+) = m \cd \la^m \cd \lim_{\ep \searrow 0} \Big( \sum_{l = 0}^\infty \frac{1}{l!} \cd a_l(\la,\ep) \Big)
= m \cd \la^m \cd \lim_{\ep \searrow 0} \Big( \sum_{l = 0}^\infty \frac{1}{l!} \cd b_l(\la,\ep) \Big)
\end{split}
\]
This proves the present theorem.
\end{proof}

\section{Integration over the flow parameter}
As mentioned in the beginning of Section \ref{s:secapphom}, the purpose of this Section is to integrate the expression for the homological index, which we obtained in Theorem \ref{t:secapphom}, over the flow parameter. This will provide us with a new formula for the homological index which only involves operators on the `inner' Hilbert $H_2$. A particular feature of this formula is also that it does no longer involve the rescaling parameter $\ep \in (0,1]$. 

Recall our notation $A_+ := h_+ \cd A = h(\infty) \cd A$ and $A_- := h_- \cd A = h(-\infty) \cd A$.

\begin{thm}\label{t:intflopar}
Let $m \in \nn$ with $m > p_0/2$ be given. There exists a constant $C > 0$ such that
\[
\begin{split}
 \T{H-Ind}_\la^m(D_+)  
& = \la^m \cd C_{m + 1/2} 
\cd \sum_{l = 0}^\infty \frac{1}{(l + 1)!} \cd 
\T{Tr}\Big( 
\frac{d^l}{dr} \Big(
A_+ \cd \big(\la +  (D_2 + r \cd A_+)^2\big)^{-m-1/2} \\
& \qq  - A_- \cd \big(\la +  (D_2 + r \cd A_-)^2\big)^{-m-1/2} \Big) \Big|_{r = 0} \Big)
\end{split}
\]
where the sum converges absolutely for all $\la \geq C$ and the constant is given by $C_{m + 1/2} := \frac{m}{\pi} \cd \int_{-\infty}^\infty (1 + \eta^2)^{-m-1} \, d\eta$.
\end{thm}

The starting point for proving Theorem \ref{t:intflopar} is the main Theorem \ref{t:secapphom} which we proved in Section \ref{s:secapphom}. This latter theorem provides a power series expansion of the homological index where each entry in the sum has the form
\[
\frac{1}{l!} \cd \T{Tr}\Big( F(\ep) \cd h_\ep^l \cd \frac{d^l}{dr} (\lap + \be + \la)^{-m-1} \big|_{r =0} \Big)
\]
for $l \in \nn_0$ and $\ep \in (0,1]$. Note that we have replaced the complex parameter $z$ by a real parameter $r$. It is therefore now our main interest to compute the operator trace appearing in the above expression. This computation will rely on the factorization of the operator trace associated to a tensor product of Hilbert spaces. This factorization result is described in Appendix \ref{ss:pretra}. For the convenience of the reader we recall that
\[
F(\ep) = 2 \cd \frac{d h_\ep}{dt} \cd A : L^2(\rr,H_2) \to L^2(\rr,H_2) \q \T{for all } \ep \in (0,1]
\]
and that $\be(r)$ is defined as the closure of
\[
r \cd (D_2 A + A D_2) + r^2 \cd A^2 : L^2(\rr,H_2)^\infty \to L^2(\rr,H_2)
\]
for all $r \in \rr$. We remark that $\be(r)$ is a selfadjoint unbounded operator and furthermore that
\[
\be(r) + \lap = \De_1 + (D_2 + r A)^2 : \sD(\lap) \to L^2(\rr,H_2)
\]
is a positive unbounded operator for all $r \in \rr$. Recall here that $\De_1 : \sD(\De_1) \to L^2(\rr,H_2)$ is defined as the closure of the Laplacian
\[
- \frac{d^2}{dt} \ot 1 : C_c^\infty(\rr) \ot H_2 \to L^2(\rr,H_2) 
\] 

Let us also introduce the unbounded selfadjoint operators $\be_+(r)$ and $\be_-(r)$ defined as the closures of
\[
\begin{split}
& r \cd (D_2 A_+ + A_+ D_2 ) + r^2 A_+^2 : H_2^\infty \to H_2 \q \T{and} \\
& r \cd (D_2 A_- + A_- D_2 ) + r^2 A_-^2 : H_2^\infty \to H_2
\end{split}
\]
respectively, for all $r \in \rr$, where we recall that $A_+ = h_+ A$ and $A_- = h_- A$.

Notice that the dense subspaces $H_2^\infty \su H_2$ and $L^2(\rr,H_2)^\infty \su L^2(\rr,H_2)$ are defined by
\[
H_2^\infty := \cap_{k = 0}^\infty \sD(\De_2^{k/2}) \q \T{and} \q L^2(\rr,H_2)^\infty := \cap_{k = 0}^\infty \sD(\lap^{k/2})
\]
respectively.

We will in this Section only consider a fixed $m \in \nn$ with $m > p_0/2$.

\begin{lemma}\label{l:laptra}
Let $l \in \nn_0$, $r \in \rr$, $\la > 0$, and $\ep \in (0,1]$ be given. Then
\[
F(\ep) \cd h_\ep^l \cd (\lap + \be(r) + \la)^{-m-1}
= \frac{1}{m!} \cd \int_0^\infty F(\ep) \cd h_\ep^l \cd s^m \cd \exp\big(-s (\lap + \be(r) + \la)\big) \, ds
\]
where the integral converges absolutely in trace norm.
\end{lemma}
\begin{proof}
We start by rewriting the power of the resolvent $(\lap + \be(r) + \la)^{-m-1}$ using the Laplace transform, thus
\[
(\lap + \be(r) + \la)^{-m-1} = \frac{1}{m!} \cd \int_0^\infty s^m \cd \exp\big( -s (\lap + \be(r) + \la)\big) \, ds
\]

Notice now that the operator
\[
F(\ep) \cd h_\ep^l \cd (\lap + \be(r) + \la)^{-m-1 + \de} : L^2(\rr,H_2) \to L^2(\rr,H_2)
\]
is of trace class for all $\de \in (0,1/2)$, see Proposition \ref{p:sumpro} and the results of Appendix \ref{s:holpseope} (recall here that $m > p_0/2$). It therefore suffices to show that the integral
\[
\int_0^\infty s^m \cd (\lap + \be(r) + \la)^{m+1 - \de} \cd \exp\big(-s (\lap + \be(r) + \la) \big) \, ds
\]
converges absolutely in operator norm for some $\de \in (0,1/2)$. But this follows by an application of the spectral theorem for selfadjoint unbounded operators.
\end{proof}

Recall that $L^2(\rr,H_2)$ may be regarded as $L^2(\rr) \hot H_2$. We may thus apply the bounded operator $\T{Tr} \ot 1 : \sL^1(L^2(\rr) \hot H_2) \to \sL^1(H_2)$ to any trace class operator $T : L^2(\rr,H_2) \to L^2(\rr,H_2)$, see Appendix \ref{ss:pretra}. We are particularly interested in the trace class operator which we analyzed in Lemma \ref{l:laptra}.

\begin{lemma}\label{l:intflopar}
Let $l \in \nn_0$, $r \in \rr$, $\la > 0$, and $\ep \in (0,1]$ be given. Then
\[
\begin{split}
& (\T{Tr} \ot 1)\big( F(\ep) \cd h_\ep^l \cd (\lap + \be(r) + \la)^{-m-1} \big) \\
& \q  = \frac{(h_+^{l+1} - h_-^{l+1})\cd A}{\pi \cd (l + 1)} \cd \int_{-\infty}^\infty (\la + (D_2 + rA)^2 + \xi^2)^{-m-1} \, d\xi
\end{split}
\]
\end{lemma}
\begin{proof}
By Lemma \ref{l:laptra} and Proposition \ref{p:intver} we have that
\[
\begin{split}
& (\T{Tr} \ot 1)\big( F(\ep) \cd h_\ep^l \cd (\lap + \be(r) + \la)^{-m-1} \big) \\
& \q = \frac{1}{m!} \cd \int_0^\infty s^m \cd (\T{Tr} \ot 1)\Big( F(\ep) \cd h_\ep^l \cd \exp\big(-s(\lap + \be(r) + \la) \big) \Big) \, ds
\end{split}
\]

Let now $s \in (0,\infty)$ be given. Then
\begin{equation}\label{eq:intflopar}
\begin{split}
& (\T{Tr} \ot 1)\big( F(\ep) \cd h_\ep^l \cd \exp(-s(\lap + \be(r) + \la) ) \big) \\
& \q = 2 A \cd \exp\big(-s(\la + (D_2 + r A)^2) \big) \cd \T{Tr}\Big( \frac{dh_\ep}{dt} \cd h_\ep^l \cd \exp(-s \De_1)\Big)
\end{split}
\end{equation}
where we have applied the identity
\[
\exp\big(-s(\lap + \be(r) + \la) \big) = \exp\big( -s(\la + (D_2 + rA)^2)\big) \cd \exp(-s \De_1)
\]

But the trace on the right hand side of \eqref{eq:intflopar} is given by
\[
\begin{split}
\T{Tr}\Big( \frac{dh_\ep}{dt} \cd h_\ep^l \cd \exp(-s \De_1)\Big)
& = \frac{1}{2 \pi} \cd \int_{-\infty}^\infty \frac{dh_\ep}{dt} \cd h_\ep^l \, dt
\cd \int_{-\infty}^\infty \exp(-s \xi^2) d\xi \\
& = \frac{h_+^{l + 1} - h_-^{l+1}}{2 \pi \cd (l + 1)} \cd \int_{-\infty}^\infty \exp(-s \xi^2) \, d\xi
\end{split}
\]

Recollecting the above identities, we obtain that
\[
\begin{split}
& (\T{Tr} \ot 1)\big( F(\ep) \cd h_\ep^l \cd (\lap + \be(r) + \la)^{-m-1} \big) \\
& \q = \frac{1}{m!} \cd \int_0^\infty \Big( s^m \cd 2 A \cd \exp\big(-s(\la + (D_2 + r A)^2) \big) \\ 
& \qqq \cd
\frac{h_+^{l + 1} - h_-^{l+1}}{2 \pi \cd (l + 1)} \cd \int_{-\infty}^\infty \exp(-s \xi^2) \, d\xi \Big) ds \\
& \q = \frac{h_+^{l + 1} - h_-^{l+1}}{\pi \cd (l + 1) \cd m!} \cd A \cd
\int_0^\infty \int_{-\infty}^\infty s^m \cd \exp\big(-s(\la + (D_2 + r A)^2 + \xi^2) \big) \, d\xi ds \\
& \q = \frac{h_+^{l + 1} - h_-^{l+1}}{\pi \cd (l + 1)} \cd A \cd \int_{-\infty}^\infty (\la + (D_2 + rA)^2 + \xi^2)^{-m-1} \, d\xi
\end{split}
\]
where we notice that the last identity follows by changing the order of integration. This operation is allowed since the integrals involved converge absolutely in operator norm. 

The lemma is now proved.
\end{proof}

We notice that the quantity in Lemma \ref{l:intflopar} is independent of the given $\ep \in (0,1]$.
The next result reduces the computation of the homological index to a computation on the `inner' Hilbert space $H_2$.
 
\begin{prop}\label{p:intflopar}
Let $l \in \nn_0$, $\la > 0$, and $\ep \in (0,1]$ be given. Then
\[
\begin{split}
& (\T{Tr} \ot 1)\Big( F(\ep) \cd h_\ep^l \cd \frac{d^l}{dr} (\lap + \be + \la)^{-m-1} \big|_{r = 0}\Big) \\
& \q =
\frac{1}{\pi \cd (l+1)} \cd \int_{-\infty}^{\infty} (1 + \eta^2)^{-m-1} \, d\eta \\
& \qq \cd \frac{d^l}{dr}\Big( A_+ (\la + (D_2 + r A_+)^2)^{-m-1/2} - A_- (\la + (D_2 + r A_-)^2)^{-m-1/2}\Big) \Big|_{r = 0} 
\end{split}
\]
\end{prop}
\begin{proof}
By Proposition \ref{p:tranorhol} we obtain that $r \mapsto F(\ep) \cd \big(\lap + \be(r) + \la \big)^{-m-1}$ is smooth in trace norm on an open neighborhood of zero. Thus, since $\T{Tr} \ot 1 : \sL^1(L^2(\rr) \hot H_2) \to \sL^1(H_2)$ is a bounded operator (by Proposition \ref{p:intver}) we have that
\[
\begin{split}
& (\T{Tr} \ot 1)\Big( F(\ep) \cd h_\ep^l \cd \frac{d^l}{dr} (\lap + \be + \la)^{-m-1} \big|_{r = 0}\Big) \\
& \q = \frac{d^l}{dr} \Big(  (\T{Tr} \ot 1)\big( 
F(\ep) \cd h_\ep^l \cd (\lap + \be + \la)^{-m-1}
\big) \Big) \Big|_{r = 0} \\
& \q = \frac{(h_+^{l+1} - h_-^{l+1}) \cd A}{\pi \cd (l + 1)} \cd \frac{d^l}{dr} \Big( \int_{-\infty}^\infty
  \big(\la + (D_2 + rA)^2 + \xi^2 \big)^{-m-1} \, d\xi \Big) \Big|_{r = 0}
\end{split}
\]
where the second identity follows by Lemma \ref{l:intflopar}.

To prove the proposition it therefore suffices to verify that
\[
\begin{split}
& h_{\pm}^l \cd \frac{d^l}{dr} \Big( \int_{-\infty}^\infty
  \big(\la + (D_2 + rA)^2 + \xi^2 \big)^{-m-1} \, d\xi \Big) \Big|_{r = 0} \\
& \q = \int_{-\infty}^\infty (1 + \eta^2)^{-m-1} \, d\eta \cd \frac{d^l}{dr}   \big(\la + (D_2 + rA_{\pm})^2 \big)^{-m-1/2} \Big|_{r =0}
\end{split}
\]
But this is the content of the next lemma.
\end{proof}

\begin{lemma}\label{l:charul}
Let $l \in \nn_0$, $\la > 0$ and $\mu \in \rr$ be given. Then we have the identity
\[
\begin{split}
& \mu^l \cd \frac{d^l}{dr} \Big( \int_{-\infty}^\infty
  \big(\la + (D_2 + rA)^2 + \xi^2 \big)^{-m-1} \, d\xi \Big) \Big|_{r = 0} \\
& \q = \int_{-\infty}^\infty (1 + \eta^2)^{-m-1} \, d\eta \cd \frac{d^l}{dr}   \big(\la + (D_2 + \mu rA)^2 \big)^{-m-1/2} \Big|_{r =0}
\end{split}
\]
\end{lemma}
\begin{proof}
By a change of variables we obtain that
\[
\int_{-\infty}^\infty
  \big(\la + (D_2 + rA)^2 + \xi^2 \big)^{-m-1} \, d\xi
= \int_{-\infty}^\infty (1 + \eta^2)^{-m-1} \, d\eta \cd \big(\la + (D_2 + rA)^2 \big)^{-m-1/2}
\]
Furthermore, by Lemma \ref{l:invhol} and Proposition \ref{p:traholrea} we conclude that the map
\[
f : \rr \to \sL(H_2) \q f : r \mapsto (\la + (D_2 + rA)^2)^{-m-1/2}
\]
is smooth in operator norm. It therefore follows by an application of the chain rule that
\[
\mu^l \frac{d^l f}{dr}  = \frac{d^l f_\mu}{dr} 
\]
where the smooth map $f_\mu : \rr \to \sL(H)$ is defined by $f_\mu(r) := f(\mu \cd r)$ for all $r \in \rr$. This proves the lemma.
\end{proof}

The main Theorem \ref{t:intflopar} of this Section is now clearly a consequence of Proposition \ref{p:intflopar}, Proposition \ref{p:intver} and Theorem \ref{t:secapphom}.

\section{The integral trace formula for the homological index}
We are now ready to prove the main theorem of this paper. We need to start by proving a version of the main theorem where we assume that the scaling parameter $\la > 0$ is large. It is necessary to apply this assumption to begin with since our resolvent expansions only exist for large $\la > 0$. We will then be able to prove the general version of our theorem by noticing that both sides of the integral trace formula have analytic extensions to the open subset $\cc\sem (-\infty,0]$ of the complex plane.

\begin{prop}\label{t:intlarsca}
Let $m \in \nn$ with $m > p_0/2$ be given. There exists a constant $C > 0$ such that
\[
\begin{split}
 \T{H-Ind}_\la^m(D_+) 
 = \la^m \cd C_{m + 1/2}
\cd \int_0^1 \T{Tr}\Big( A_+ (\la + (D_2 + r \cd A_+)^2)^{-m-1/2} \\
 - A_- (\la + (D_2 + r \cd A_-)^2)^{-m-1/2} \Big) \, dr
\end{split}
\]
for all $\la \geq C$ where the constant in front is given by 
$C_{m + 1/2} := \frac{m}{\pi} \cd \int_{-\infty}^{\infty} (1 + \eta^2)^{-m-1} \, d\eta$.
\end{prop}
\begin{proof}
By Proposition \ref{p:tranorhol} and Proposition \ref{p:traholrea} there exists a constant $C_0 > 0$ such that the maps
\[
f_{\pm}(\la, \cd) : \B B_2(0) \to \sL^1(H_2) \q f_{\pm}(\la, z) :=  A_{\pm} (\la + (D_2 + z A_{\pm})^2)^{-m-1/2}
\]
are holomorphic for all $\la \geq C_0$. It therefore follows that
\[
\begin{split}
\int_0^1 \T{Tr}\big( f_{\pm}(\la, r) \big) \, dr 
& = \sum_{l = 0}^\infty \frac{1}{(l + 1)!} 
\T{Tr}\big( \frac{d^l}{dz} f_{\pm}(\la, \cd) \big|_{z = 0} \big) \\
& = \sum_{l = 0}^\infty \frac{1}{(l + 1)!} \T{Tr}\big( \frac{d^l}{dr} f_{\pm}(\la,\cd) \big|_{r = 0} \big)
\end{split}
\]
for all $\la \geq C_0$, where $z \in \B B_2(0)$ is a complex parameter and $r \in (-2,2)$ is a real parameter.

However, by Theorem \ref{t:intflopar} there exists a constant $C_1 > 0$ such that
\[
\T{H-Ind}_\la^m(D_+) = \la^m \cd C_{m + 1/2} \cd \sum_{l = 0}^\infty \frac{1}{(l + 1)!} \T{Tr}\Big( \frac{d^l}{dr}\big(f_+(\la,\cd) - f_-(\la,\cd)\big)\big|_{r =0} \Big)
\]
for all $\la \geq C_1$.

Thus, for all $\la \geq \T{max}\{ C_0, C_1 \}$ we obtain that
\[
\T{H-Ind}_\la^m(D_+) = \la^m \cd C_{m + 1/2} \cd \int_0^1 \T{Tr}\big(f_+(\la,r) - f_-(\la,r) \big)\, dr
\]
This proves the theorem.
\end{proof}

As promised in the beginning of this Section we may remove the constraint on the scaling parameter $\la > 0$ present in Proposition \ref{t:intlarsca}. This leads us to the main theorem of this paper.

\begin{thm}
Let $m \in \nn$ with $m > p_0/2$ be given. Then
\[
\begin{split}
\T{H-Ind}_\la^m(D_+) 
= \la^m \cd C_{m + 1/2}
\cd \int_0^1 \T{Tr}\Big( A_+ (\la + (D_2 + r \cd A_+)^2)^{-m-1/2} \\
 - A_- (\la + (D_2 + r \cd A_-)^2)^{-m-1/2} \Big) \, dr
\end{split}
\]
for all $\la > 0$ where the constant is given by $C_{m + 1/2} := \frac{m}{\pi} \cd \int_{-\infty}^{\infty} (1 + \eta^2)^{-m-1} \, d\eta$.
\end{thm}
\begin{proof}
By \cite[Proposition 3.3]{CaGrKa:AOS} we have that the function
\[
\la \mapsto \T{H-Ind}_\la^m(D_+) \q (0,\infty) \to \cc
\]
has a holomorphic extension to $\cc\sem (-\infty,0]$.

Furthermore, by Proposition \ref{p:tranorhol} and Proposition \ref{p:traholrea} (with the compact Hausdorff space $X = [0,1]$) we know that the maps
\[
\la \mapsto \la^m \cd C_{m + 1/2} \int_0^1 A_{\pm} (\la + (D_2 + r \cd A_{\pm})^2)^{-m-1/2} \, dr \q
(0,\infty) \to \sL^1(H_2)
\]
have extensions to $\cc\sem (-\infty,0]$ which are holomorphic in trace norm. As a consequence we obtain that the functions
\[
\la \mapsto \la^m \cd C_{m + 1/2} \int_0^1 \T{Tr}\Big( A_{\pm} (\la + (D_2 + r \cd A_{\pm})^2)^{-m-1/2} \Big) \, dr \q
(0,\infty) \to \cc
\]
have holomorphic extensions to $\cc\sem (-\infty,0]$.

By Theorem \ref{t:intlarsca} there exists a constant $C > 0$ such that
\[
\begin{split}
& \T{H-Ind}_\la^m(D_+) 
 = \la^m \cd C_{m + 1/2}
\cd \int_0^1 \T{Tr}\Big( A_+ (\la + (D_2 + r \cd A_+)^2)^{-m-1/2} \\
& \qqq \qq \qq - A_- (\la + (D_2 + r \cd A_-)^2)^{-m-1/2} \Big) \, dr
\end{split}
\]
for all $\la \geq C$. The result of the present theorem therefore follows by the uniqueness of holomorphic extensions.
\end{proof}

\appendix{

\section{Holomorphic families of pseudodifferential operators}\label{s:holpseope}
%
In this appendix we are concerned with holomorphic families of abstract pseudodifferential operators. This material will be used throughout this text and we therefore expose it in a careful way. We will work in the general context where we have a fixed positive unbounded operator $\De : \sD(\De) \to H$ which acts on a separable Hilbert space $H$.

Recall the notation $H^\infty := \cap_{k \in \nn_0} \sD(\De^{k/2})$ of domains of the powers of $\De$. 
%

\begin{dfn}
Let $n \in \zz$. A \emph{pseudodifferential operator} of order $n$ is a \emph{closed} unbounded operator $T : \sD(T) \to H$ such that
\begin{enumerate}
\item $H^\infty \su \sD(T)$;
\item $T(\xi) \in H^\infty$ for all $\xi \in H^\infty$;
\item $(1 + \De)^{-n/2} \de^k(T) : H^\infty \to H$ extends to a bounded operator on $H$ for all $k \in \nn_0$, where $\de = [\De^{1/2}, \cd]$.
\end{enumerate}
The set of pseudodifferential operators of order $n$ is denoted by $\T{OP}^n(\De)$.
\end{dfn}

The sum of two pseudodifferential operators $S \in \T{OP}^n(\De)$ and $T \in \T{OP}^m(\De)$ is defined as the closure of the restriction $S + T : H^\infty \to H$. The product of $S$ and $T$ is defined as the closure of
$ST : H^\infty \to H$. It can then be proved that these algebraic operations provide the union
\[
\T{OP}^\infty(\De) := \bigcup_{n \in \zz} \T{OP}^n(\De)
\]
with the structure of a unital algebra over $\cc$.
A lot more can be said about the general structure of pseudodifferential operators, and we refer the interested reader to the papers, \cite{CoMo:LIF,Hig:LFN,CGRS:ILN}.

In this appendix we are mainly interested in \emph{holomorphic families} of pseudodifferential operators. {\bla To ease the exposition we will fix an open set $U \su \cc$ and a compact Hausdorff space $X$.}

\begin{dfn}\label{d:holfampse}
Let $n \in \zz$. A map $\al : U \ti X \to \T{OP}^n(\De)$ is said to be \emph{holomorphic in supremum norm} when the following holds for each $k \in \nn_0$:
\begin{enumerate}
\item The map
\[
U \ti X \to \sL(H) \q (z,x) \mapsto (1 + \De)^{-n/2} \de^k(\al(z,x))
\]
is continuous in operator norm.
\item The map
\[
U \to C\big( X; \sL(H) \big) \q z \mapsto (1 + \De)^{-n/2} \de^k(\al(z,\cd))
\]
is holomorphic with respect to the supremum norm
\[
\| \cd \|_\infty : C\big( X; \sL(H) \big) \to [0,\infty) \q
f \mapsto \sup_{x \in X} \| f(x) \|
\]
\end{enumerate}
\end{dfn}

We now prove a few lemmas which are concerned with the multiplicative structure of maps which are holomorphic in supremum norm.

\begin{lemma}\label{l:prohol}
Let $\al : U \ti X \to \T{OP}^n(\De)$ and $\be : U \ti X \to \T{OP}^m(\De)$ be holomorphic in supremum norm for some $m,n\in \zz$. Then the pointwise product $\al \cd \be : U \ti X \to \T{OP}^{n + m}(\De)$ is holomorphic in supremum norm.
\end{lemma}
\begin{proof}
Let $(z,x) \in U \ti X$ and $k \in \nn_0$ be given. 

Suppose first that $m \leq 0$. We then have that
\[
\begin{split}
& (1 + \De^{1/2})^{-n-m}\de^k\big( \al(z,x) \be(z,x) \big)(\xi) \\
& \q = \sum_{i = 0}^k {k \choose i} (1 + \De^{1/2})^{-n-m} \cd \de^i(\al(z,x)) \cd \de^{k - i}(\al(z,x))(\xi) \\
& \q = \sum_{i = 0}^k \sum_{j = 0}^m {k \choose i} {m \choose j} 
(1 + \De^{1/2})^{-n} \cd \de^{i+j}(\al(z,x)) \\ 
& \qqq \qqq \cd (1 + \De^{1/2})^{-m - j}\de^{k - i}(\al(z,x))(\xi)
\end{split}
\]
for all $\xi \in H^\infty$. This computation proves the claim of the lemma in this case.

Suppose then that $m \geq 0$. We then have that
\[
\begin{split}
& (1 + \De^{1/2})^{-n-m}\de^k\big( \al(z,x) \be(z,x) \big)(\xi) \\
& \q = \sum_{i = 0}^k {k \choose i} (1 + \De^{1/2})^{-n-m} \cd \de^i(\al(z,x)) \cd \de^{k - i}(\al(z,x))(\xi) \\
& \q = \sum_{i = 0}^k \sum_{j = 0}^m {k \choose i} {m \choose j} \cd (-1)^j \cd
(1 + \De^{1/2})^{-n-j} \cd \de^{i+j}(\al(z,x)) \\ 
& \qqq \qqq \cd (1 + \De^{1/2})^{-m}\de^{k - i}(\al(z,x))(\xi)
\end{split}
\]
for all $\xi \in H^\infty$. This ends the proof of the lemma.
\end{proof}

\begin{lemma}\label{l:invhol}
Let $\al : U \ti X \to \T{OP}^1(\De)$ be holomorphic in supremum norm. Suppose that $\De + \al(z,x) : \sD(\De) \to H$ is invertible as a closed unbounded operator for each $(z,x) \in U \ti X$. Then the map
\[
(\De + \al)^{-1} : U \ti X \to \T{OP}^{-2}(\De) \q (z,x) \mapsto (\De + \al(z,x))^{-1}
\]
is holomorphic in supremum norm.
\end{lemma}
\begin{proof}
Remark first of all that $\sD(\De^{1/2}) \su \sD(\al(z,x))$ for all $(z,x) \in U \ti X$. To see this, let $\xi \in \sD(\De^{1/2})$ and choose a sequence $\{\xi_n\}$ in $H^\infty$ with $\xi_n \to \xi$ and $\De^{1/2} \xi_n \to \De^{1/2} \xi$. Then $\al(z,x) \xi_n \to \al(z,x)(\xi)$. Indeed, this follows since
\[
\begin{split}
\al(z,x)(\xi_n) & = (1 + \De^{1/2})^{-1} (1 + \De^{1/2}) \al(z,x)(\xi_n) \\
& = (1 + \De^{1/2})^{-1} \cd \big( \al(z,x) + \de(\al(z,x)) + \al(z,x) \De^{1/2} \big)(\xi_n)
\end{split}
\]
The unbounded operator $\De + \al(z,x) : \sD(\De) \to H$ is thus well-defined for all $(z,x) \in U \ti X$. The fact that $\De + \al(z,x) : \sD(\De) \to H$ is closed is a consequence of the Kato-Rellich theorem, see \cite[Chapther IV, Theorem 1.1]{Kat:PT}.

It is now straightforward to verify that the map
\[
\De + \al : U \ti X \to \T{OP}^2(\De) \q (z,x) \mapsto \De + \al(z,x)
\]
is holomorphic in supremum norm.

Let $(z,x) \in U \ti X$ and let us continue by showing that
$
(\De + \al(z,x))^{-1} \in \T{OP}^{-2}(\De).
$
To this end, it is clearly sufficient to prove that
\[
(\De + 1)(\De + \al(z,x))^{-1} \in \T{OP}^0(\De).
\]
We therefore notice that the linear map
\[
\be(z,x) := (\De + \al(z,x))(\De + 1)^{-1} : H \to H
\]
is bijective. Furthermore, we have that $\be(z,x) \in \T{OP}^0(\De)$. Since $\T{OP}^0(\De)$ is closed under holomorphic functional calculus (as a unital subalgebra of $\sL(H)$), see \cite[Lemma 2.3]{CaGrKa:AOS}, we may conclude that
\[
(\De + 1)(\De + \al(z,x))^{-1} = \big( (\De + \al(z,x))(\De + 1)^{-1} \big)^{-1} \in \T{OP}^0(\De). 
\]

To see that
\[
U \ti X \to \T{OP}^{-2}(\De) \q (z,x) \mapsto (\De + \al(z,x))^{-1}
\]
is holomorphic in supremum norm it is now enough to check that the map
\[
\be^{-1} : U \ti X \to \T{OP}^0(\De) \q (z,x) \mapsto (\De + 1)(\De + \al(z,x))^{-1}
\]
is holomorphic in supremum norm. We remark that it follows by Lemma \ref{l:prohol} that the map
\[
\be : U \ti X \to \T{OP}^0(\De) \q (z,x) \mapsto (\De + \al(z,x))(\De + 1)^{-1}
\]
is holomorphic in supremum norm. 

We now prove by induction on $k \in \nn_0$ that $\be^{-1}$ satisfies the conditions $(1)$ and $(2)$ in Definition \ref{d:holfampse}.
For $k=0$ this follows from the identity
\[
\be(z,x)^{-1} - \be(w,y)^{-1} = - \be(z,x)^{-1} (\be(z,x) - \be(w,y)) \be(w,y)^{-1} 
\]
which is valid for all $(z,x) ,(w,y) \in U \ti X$.

To prove the induction step, we let $k \in \nn$ and note that
\[
\begin{split}
 \de^k\big( \be(z,x)^{-1} \big)(\xi) 
 &= -\de^{k-1}\big(  \be(z,x)^{-1} \de(\be(z,x)) \be(z,x)^{-1} \big)(\xi) \\
&  = - \sum_{i = 0}^{k-1} \sum_{j = 0}^{k-1-i} {k-1 \choose i} {k-1-i \choose j} \\
& \qqq \cd \de^i\big( \be(z,x)^{-1} \big)  \cd \de^{j + 1}(\be(z,x)) \cd 
\de^{k-1-i - j}\big( \be(z,x)^{-1} \big)(\xi)
\end{split}
\]
for all $(z,x) \in U \ti X$ and all $\xi \in H^\infty$. 
\end{proof}

We end this appendix by providing a useful estimate.

\begin{prop}\label{p:resestpow}
Let $\al : U \ti X \to \T{OP}^1(\De)$ be holomorphic in supremum norm. Suppose that, {\bla for each $k \in \nn_0$, we have
\begin{equation}\label{eq:openor}
\sup_{(z,x) \in U \ti X} \big\| \de^k\big( \al(z,x) \big) (1 + \De)^{-1/2} \big\| < \infty
\end{equation}
Then,} for each $k \in \nn_0$ and $\ep > 0$, there exists a constant $C_{k,\ep} > 0$ such that
\[
\sup_{(z,x) \in U \ti X}\big\| (\De + 1)^p \cd \de^k \big( (\De + \al(z,x) + \la)^{-1} \big) \big\| \leq C_{k,\ep} \cd \la^{-1 + p}
\]
for all $p \in [0,1]$ and all $\la \geq 1$ with
\[
\sqrt{\la} \geq \sup_{(z,x) \in U \ti X} \| \al(z,x) (\De + 1)^{-1/2} \| + \ep
\]
\end{prop}
\begin{proof}
Notice first that $\de^k\big( \al(z,x) \big) (\De + 1)^{-1/2} : H^\infty \to H$ extends to a bounded operator for all $(z,x) \in U \ti X$ and all $k \in \nn_0$. The operator norm appearing in \eqref{eq:openor} therefore makes sense.

Notice next that the closed unbounded operator
\[
\De + \al(z,x) + \la : \sD(\De) \to H
\]
is invertible for all $(z,x) \in U \ti X$ and all $\la \geq 1$ with
\begin{equation}\label{eq:deflam}
\sup_{(z,x) \in U \ti X} \| \al(z,x) (\De + 1)^{-1/2} \| < \sqrt{\la}
\end{equation}
Indeed, in this case the inverse is given by
\begin{equation}\label{eq:resalplam}
(\De + \al(z,x) + \la)^{-1} = (\De + \la)^{-1} \big( 1 +  \al(z,x) (\De + \la)^{-1} \big)^{-1}
\end{equation}
Remark here that
\[
\big\| \al(z,x) (\De + \la)^{-1} \big\| \leq \la^{-1/2} \cd \| \al(z,x) (\De + 1)^{-1/2} \| < 1
\]
for all $(z,x) \in U \ti X$, where the first inequality follows since $\la \geq 1$ and the second follows by \eqref{eq:deflam}.

To ease the notation, let now
\[
L := \sup_{(z,x) \in U} \| \al(z,x) (\De + 1)^{-1/2} \|
\]
Let $\ep > 0$, $p \in [0,1]$ and $\la \geq 1$ with
$
\sqrt{\la} \geq L  + \ep
$
be given. Using the explicit form of the inverse in \eqref{eq:resalplam} we obtain that
\[
(\De + 1)^p (\De + \al(z,x) + \la)^{-1} = (\De + 1)^p (\De + \la)^{-1} \cd \big( 1 + \al(z,x) (\De + \la)^{-1} \big)^{-1}
\]
for all $(z,x) \in U \ti X$. And it therefore follows that
\[
\begin{split}
\big\| (\De + 1)^p (\De + \al(z,x) + \la)^{-1} \big\| 
& \leq \big\| (\De + \la)^{-1 + p} \big\|
\cd \big\| (1 + \al(z,x)(\De + \la)^{-1})^{-1}\big\| \\
& \leq \la^{-1 + p} \cd \big( 1 - \| \al(z,x)(\De + \la)^{-1} \| \big)^{-1} \\
& \leq \la^{-1 + p} \cd \big( 1 - {L}{(L + \ep)^{-1}} \big)^{-1}
\end{split}
\]
where the second inequality is a consequence of the series expansion for $(1 + \al(z,x)(\De + \la)^{-1})^{-1}$ and the third inequality follows since
\[
\| \al(z,x) (\De + \la)^{-1} \| \leq \la^{-1/2} \cd L \leq {L}{(L + \ep)^{-1}}
\]
for all $(z,x) \in U \ti X$. These considerations prove the proposition in the case where $k = 0$.

The general case now follows by induction on $k \in \nn_0$ because of the identity
\[
\begin{split}
 \de^k\big( (\De + \al + \la)^{-1} \big) 
& = - \sum_{i = 0}^{k - 1} \sum_{j = 0}^{k - 1 - i} {k - 1 \choose i} {k - 1 - i \choose j} \\
& \q \  \cd \de^i\big( (\De + \al + \la)^{-1} \big) \cd \de^{j + 1}(\al) 
\cd \de^{k - 1 - i - j}\big( (\De + \al + \la)^{-1} \big)
\end{split}
\]
which was derived in the proof of Lemma \ref{l:invhol}.
\end{proof}

\section{Trace norm holomorphy}\label{s:tranorhol}
In this appendix we will draw some consequences of the investigations on holomorphic families of abstract pseudodifferential operators which we carried out in Appendix \ref{s:holpseope}. We will here mainly be concerned with holomorphic maps with values in the trace class operators. The general context will be as follows: 

Throughout this appendix we let $T : H \to H$ be a bounded operator and $\De : \sD(\De) \to H$ be a selfadjoint positive unbounded operator on the separable Hilbert space $H$. We will suppose that there exists an $m \in \nn_0$ such that
\[
T \cd (1 + \De)^{-m} \in \sL^1(H)
\]
{\bla We will also fix an open set $U \su \cc$ and a compact Hausdorff space $X$.}

The next result follows from Lemma \ref{l:invhol} and Lemma \ref{l:prohol} together with the basic inequality
$
\| R \cd S \|_1 \leq \| R \|_1 \cd \| S \|
$
which is valid for all $R \in \sL^1(H)$ and all $S \in \sL(H)$, where $\| \cd \|_1 : \sL^1(H) \to [0,\infty)$ denotes the trace norm.

\begin{prop}\label{p:tranorhol}
Let $\al : U \ti X \to \T{OP}^1(\De)$ be a map which is holomorphic in supremum norm. Suppose that
\[
\De + \al(z,x) : \sD(\De) \to H
\]
is invertible (as a closed unbounded operator) for all $(z,x) \in U \ti X$. Then the map
\[
U \ti X \to \sL^1(H) \q (z,x) \mapsto T \cd (\De + \al(z,x))^{-m}
\]
is continuous in trace norm. Furthermore, the map
\[
U \to C\big( X; \sL^1(H) \big) \q z \mapsto T \cd (\De + \al(z,\cd))^{-m}
\]
is holomorphic with respect to the norm
\begin{equation}\label{eq:norcontra}
\| \cd \|_{\infty,1} : C\big( X; \sL^1(H) \big) \to [0,\infty) \q
f \mapsto \sup_{x \in X} \| f(x) \|_1
\end{equation}
\end{prop}

For each $r > 0$, let $\B B_r(0) := \{ z \in \cc \, \big| \, |z| < r \}$ {\bla denote the open ball of radius $r > 0$ and center $0 \in \cc$.} As a consequence of Proposition \ref{p:tranorhol} and standard results about holomorphic Banach space valued maps we obtain the following:

\begin{prop}\label{p:tranorpow}
Let $\de > 0$ and let $\al : \B B_{1 + 2\de}(0) \ti X \to \T{OP}^1(\De)$ be holomorphic in supremum norm. Suppose that
\[
\De + \al(z,x) : \sD(\De) \to H
\]
be invertible (as a closed unbounded operator) for all $(z,x) \in U \ti X$. Then there exists a sequence $\{a_j\}_{j = 0}^\infty$ in $C(X;\sL^1(H))$ such that the identity
\[
T \cd (\De + \al(z,\cd ) )^{-m} = \sum_{j = 0}^\infty z^j \cd a_j
\]
holds for all $z \in \B B_{1 + \de}(0)$ and the sum converges absolutely in the norm on $C(X; \sL^1(H))$ described in \eqref{eq:norcontra}.  
\end{prop}

Under the conditions of the above proposition we may express the coefficients $a_j \in C\big(X; \sL^1(H) \big)$ as
\[
a_j(x) = \frac{T}{j!} \cd \frac{d^j}{dz} (\De + \al(\cd , x)  )^{-m} \big|_{z = 0} \q \forall x \in X
\]
where the derivatives are taken with respect to the operator norm.

The next result allows us to improve the statement of Proposition \ref{p:tranorhol} by including non-integer exponents:

\begin{prop}\label{p:traholrea}
Let $\al : U \ti X \to \T{OP}^1(\De)$ be holomorphic in supremum norm. Suppose that
\[
\T{Sp}\big( \De + \al(z,x) \big) \su \cc\sem (-\infty,0]
\]
for all $(z,x) \in U \ti X$, where $\T{Sp}\big( \De + \al(z,x)\big)$ denotes the spectrum of $\De + \al(z,x) : \sD(\De) \to H$ as a closed unbounded operator. Suppose furthermore that
\[
\De + \al(t,x) : \sD(\De) \to H
\]
is selfadjoint and positive for all $(t,x) \in U \cap \rr \ti X$. Let $q \in (0,1)$. Then the map
\[
U \cap \rr \ti X \to \sL(H) \q (t,x) \mapsto (\De + \al(t,x))^{-q}
\]
has a continuous extension $\be : U \ti X \to \sL(H)$ such that
\[
U \to C(X;\sL(H)) \q z \mapsto \be(z,\cd)
\]
is holomorphic in the $C^*$-norm on $C(X;\sL(H))$.
\end{prop}
\begin{proof}
To ease the notation, we define
\[
R(z,x,\mu) := (\De + \al(z,x) + \mu)^{-1} \q \T{for all } (z,x,\mu) \in U \ti X \ti [0,\infty)
\]
We notice that it follows by Lemma \ref{l:invhol} that the map
\[
U \ti X \ti [0,\infty) \to \sL(H) \q (z,x,\mu) \mapsto (\De + 1)^p R(z,x,\mu)
\]
is continuous in operator norm for all $p \in [0,1]$.

Let us now recall the integral formula
\[
(\De + \al(t,x))^{-q} = \frac{\sin(q\pi)}{\pi} \cd \int_0^\infty \mu^{-q} \cd R(t,x,\mu) \, d\mu
\]
which is valid for all $(t,x) \in U \cap \rr \ti X$. We therefore define the map
\[
\be : U \ti X \to \sL(H) \q 
(z,x) \mapsto \frac{\sin(q\pi)}{\pi} \cd \int_0^\infty \mu^{-q} \cd R(z,x,\mu) \, d\mu 
\]

Consider now a fixed element $(z_0,x_0) \in U \ti X$ and choose an open set $V \su U$ such that $z_0 \in V$ and such that the closure $\ov{V}$ is compact and contained in $U$.

By Lemma \ref{l:invhol} and Proposition \ref{p:resestpow} we may then find a constant $C > 0$ such that
\begin{equation}\label{eq:estres}
\| R(z,x,\mu) \| \leq C \cd (1 + \mu)^{-1} \q \T{and} \q
\|(\De + 1)^{1/2} R(z,x,\mu) \| \leq C \cd (1 + \mu)^{-1/2}
\end{equation}
for all $(z,x,\mu) \in V \ti X \ti [0,\infty)$.

The estimates in \eqref{eq:estres} clearly imply that the integral
\[
\int_0^\infty \mu^{-q} R(z_0,x_0,\mu) \, d\mu
\]
converges absolutely in operator norm. The map $\be : U \ti X \to \sL(H)$ is therefore well-defined.

To see that $\be : U \ti X \to \sL(H)$ is continuous in operator norm we apply the resolvent identity together with the estimates in \eqref{eq:estres} to obtain that
\[
\begin{split}
& \int_0^\infty \mu^{-q} \| R(z,x,\mu) - R(z_0,x_0,\mu)\| \, d\mu  \\
& \q \leq \| \big( \al(z,x) - \al(z_0,x_0)\big)(\De + 1)^{-1/2} \| \cd
\int_0^\infty \mu^{-q} \cd (1 + \mu)^{-3/2} \cd C^2 \, d\mu
\end{split}
\]
for all $(z,x) \in V \ti X$.

To see that the map
\[
U \to C(X;\sL(H)) \q z \mapsto \be(z,\cd)
\]
is holomorphic in the norm $\| \cd \|_\infty : f \mapsto \sup_{x \in X} \| f(x) \|$ we define the bounded operator
\[
\ga(z_0,x) := \frac{d\big( \al(\cd,x)(\De + 1)^{-1/2}\big)}{dz} \big|_{z = z_0}
\]
for all $x \in X$. Another application of the resolvent identity and the estimates in \eqref{eq:estres} then shows that
\[
\begin{split}
& \int_0^\infty \mu^{-q} \cd \big\| \frac{R(z,x,\mu) - R(z_0,x,\mu) }{z - z_0} \\
& \qqqq + R(z_0,x,\mu) \cd \ga(z_0,x) \cd (\De + 1)^{1/2} R(z_0,x,\mu) \big\| \, d\mu \\
& \q \leq 
\big\| \frac{\al(z,x) - \al(z_0,x)}{z - z_0} \cd (\De + 1)^{-1/2} - \ga(z_0,x) \big\|
\cd \int_0^\infty \mu^{-q} \cd (1 + \mu)^{-3/2} \cd C^2 \, d\mu \\
& \qq 
+ \| \ga(z_0,x) \| \cd \| (\al(z,x) - \al(z_0,x))(\De + 1)^{-1/2} \| \\
& \qqq \cd \int_0^\infty \mu^{-q} \cd (1 + \mu)^{-2} \cd C^3 \, d\mu
\end{split}
\]
for all $(z,x) \in V \ti X$ with $z \neq z_0$. 
This proves the present proposition.
\end{proof}


\section{Factorization of operator traces}\label{ss:pretra}
Throughout this appendix we let $H_1$ and $H_2$ be separable Hilbert spaces. We denote their Hilbert space tensor product by $H_1 \hot H_2$.

We remark that the algebraic tensor product of trace ideals $\sL^1(H_1) \ot \sL^1(H_2)$ is dense in the trace ideal $\sL^1(H_1 \hot H_2)$ with respect to the trace norm. Furthermore, we may compute the restriction of the operator trace $\T{Tr} : \sL^1(H_1) \ot \sL^1(H_2) \to \cc$ as the composition
\[
\begin{CD}
\T{Tr} : \q \sL^1(H_1) \ot \sL^1(H_2) @>{\T{Tr} \ot 1}>> \sL^1(H_2) @>{\T{Tr}}>> \cc
\end{CD}
\]

\emph{The aim of this appendix is to show that $\T{Tr} \ot 1 : \sL^1(H_1) \ot \sL^1(H_2) \to \sL^1(H_2)$ extends by continuity to a bounded (contractive) operator $\T{Tr} \ot 1 : \sL^1(H_1 \hot H_2) \to \sL^1(H_2)$.}

We believe that this result is well-known by the experts, but for lack of a reference and because of the importance of this result for the present work, we have decided to supply a full proof.

Let us fix an orthonormal basis $\{e_n\}_{n = 1}^\infty$ for the Hilbert space $H_1$. For each $n \in \nn$, we let $P_n : H_1 \to H_1$ denote the orthogonal projection with image $\cc e_n \su H_1$. For each $N \in \nn$, we define $Q_N := \sum_{n=1}^N P_n : H_1 \to H_1$.

\begin{lemma}\label{l:preine}
For each $N \in \nn$ and each $R = \sum_{j = 1}^k T_j \ot S_j \in \sL^1(H_1) \ot \sL^1(H_2)$ we have the inequality of trace norms:
\[
\|(\T{Tr} \ot 1)(R(Q_N \ot 1)) \|_1 \leq \| \sum_{n = 1}^N (P_n \ot 1) R (P_n \ot 1) \|_1 
\]
\end{lemma}
\begin{proof}
For each $j \in \{1,\ldots,k\}$ and each $n \in \{1,\ldots,N\}$, define the complex number
\[
\la_j^n := \inn{e_n,T_j e_n}
\]
It follows that
\[
\T{Tr}(T_j Q_N) = \sum_{n = 1}^N \la_j^n \q \T{and} \q
P_n T_j P_n = \la_j^n P_n
\]
for all $j \in \{1,\ldots,k\}$ and all $n \in \{1,\ldots,N\}$.

This observation implies that
\[
\begin{split}
& \| (\T{Tr} \ot 1)(R(Q_N \ot 1)) \|_1  = \| \sum_{j = 1}^k \sum_{n = 1}^N \la_j^n \cd S_j \|_1
\leq \sum_{n= 1}^N \| \sum_{j = 1}^k \la_j^n \cd S_j \|_1 \\
& \q = \T{Tr}\big( \sum_{n = 1}^N P_n \ot \big|\sum_{j = 1}^k \la_j^n \cd S_j \big| \big)
= \T{Tr}\big( \big| \sum_{n = 1}^N P_n \ot \sum_{j = 1}^k \la_j^n \cd S_j \big| \big) \\
& \q = \| \sum_{j = 1}^k \sum_{n = 1}^N P_n T_j P_n \ot S_j \|_1
\end{split}
\]
The lemma is proved.
\end{proof}

\begin{prop}\label{p:intver}
The linear map
\[
\T{Tr} \ot 1 : \sL^1(H_1) \ot \sL^1(H_2) \to \sL^1(H_2)
\]
extends to a bounded contractive operator
\[
\T{Tr} \ot 1 : \sL^1(H_1 \hot H_2) \to \sL^1(H_2)
\]
Furthermore, the operator trace on $\sL^1(H_1 \hot H_2)$ can be rewritten as
\[
\T{Tr} = \T{Tr} \ci (\T{Tr} \ot 1) : \sL^1(H_1 \hot H_2) \to \cc
\]
\end{prop}
\begin{proof}
By the observations in the beginning of this appendix, it is enough to show that
\[
\| (\T{Tr} \ot 1)(R) \|_1 \leq \| R \|_1
\]
for all $R \in \sL^1(H_1) \ot \sL^1(H_2)$.

Thus, let $R \in \sL^1(H_1) \ot \sL^1(H_2)$ be given.
By \cite[Chapter III, Theorem 4.2]{GoKr:ITN} we have the inequality
\[
\| \sum_{n = 1}^N (P_n \ot 1) R (P_n \ot 1) \|_1 \leq \| R \|_1
\]
for all $N \in \nn$. Thus, by Lemma \ref{l:preine} we obtain that
\[
\| (\T{Tr} \ot 1)(R (Q_N \ot 1)) \|_1 \leq \| R\|_1
\]
for all $N \in \nn$. But this proves the present proposition since the sequence $\{ (\T{Tr} \ot 1)(R (Q_N \ot 1)) \}_{N = 1}^\infty$ converges in trace norm to $(\T{Tr} \ot 1)(R)$ by normality of the operator trace. 
\end{proof}

}

\newcommand{\etalchar}[1]{$^{#1}$}
\def\cprime{$'$}
\providecommand{\bysame}{\leavevmode\hbox to3em{\hrulefill}\thinspace}
\providecommand{\MR}{\relax\ifhmode\unskip\space\fi MR }
\providecommand{\MRhref}[2]{%
  \href{http://www.ams.org/mathscinet-getitem?mr=#1}{#2}
}
\providecommand{\href}[2]{#2}


\end{document}